\documentclass[11pt,reqno]{amsart}
\usepackage[a4paper]{geometry}
\usepackage{amsmath,amssymb}
\usepackage{mathrsfs}
\usepackage{graphicx,color}

\newtheorem{lem}{Lemma}[section]
\newtheorem{thm}[lem]{Theorem}
\newtheorem{prop}[lem]{Proposition}
\newtheorem{cor}[lem]{Corollary}
\theoremstyle{remark}
\newtheorem{expl}{Example}[section]
\numberwithin{equation}{section}
\theoremstyle{definition}
\newtheorem{df}{Definition}
\newtheorem{rem}{Remark}[section]

\def\R{\mathbb R}
\def\C{\mathbb C}
\def\D{\mathrm D}
\def\d{\mathrm d}
\def\S{\mathbb S}
\def\i{\mathrm i}
\def\e{\mathrm e}
\let\Im\relax
\DeclareMathOperator{\Im}{Im}
\let\Re\relax
\DeclareMathOperator{\Re}{Re}
\DeclareMathOperator{\spec}{spec}
\DeclareMathOperator{\diag}{diag}
\DeclareMathOperator{\bdiag}{b-diag}
\DeclareMathOperator{\trace}{tr}
\DeclareMathOperator{\supp}{supp}
\DeclareMathOperator{\spann}{span}

\begin{document}
\title[Anisotropic thermo-elasticity]{Thermo-elasticity for anisotropic media\\ in higher dimensions}\author{Jens Wirth}
\address{J. Wirth, Institut f\"ur Analysis, Dynamik und Modellierung, Fachbereich Mathematik, Universit\"at Stuttgart, Pfaffenwaldring 57, 70569 Stuttgart, Germany}
\email{jens.wirth@iadm.uni-stuttgart.de}
\date{}
\begin{abstract}
In this note we develop tools to study the Cauchy problem for the system of thermo-elasticity in higher
dimensions. The theory is developed for general homogeneous anisotropic media under non-degeneracy conditions. 

For degenerate cases a method of treatment is sketched and for the cases of cubic media and hexagonal media 
detailed studies are provided.
\end{abstract}

\maketitle

\section{Introduction}
While isotropic thermo-elasticity is a well-known and well-established subject 
(see, e.g., the book of Jiang Song-Racke \cite{JR} and references therein) only 
very few results are available for the case of anisotropic media. Among them are 
the theses of Borkenstein \cite{Bor93} for cubic media and Doll \cite{Dol04} for 
the case of rhombic media together with the authors treatments \cite{RW07}, 
\cite{Wirth:2007c}, all in two space dimensions.

In this paper the system of anisotropic thermo-elasticity in three (and more) dimensions, i.e.,
\begin{subequations}\label{eq:CP}
\begin{align}
 U_{tt} + A(\D) U +{\gamma}\nabla\theta &=0,\\
 \theta_t-\kappa\Delta\theta +{\gamma}\nabla\cdot U_t&=0
\end{align}
\end{subequations}
for the elastic displacement $U(t,\cdot):\R^n\to\R^n$ and temperature difference 
$\theta(t,\cdot):\R^n\to\R$ to the equilibrium state, will be considered. The 
system \eqref{eq:CP} couples the hyperbolic elasticity equation with the parabolic 
heat equation. The operator $A(\D)$ describes the elastic properties of the 
underlying medium, while $\kappa$ denotes its thermal conductibility. The constant 
$\gamma$ describes the thermo-elastic coupling. Basic assumptions of our theory 
are $\kappa>0$, $\gamma^2>0$ together with
\begin{itemize}
\item $A(\xi)=|\xi|^2A(\eta)$, $\eta=\xi/|\xi|$, is a 2-homogeneous matrix-valued symbol;
\item $A:\S^{n-1}\to\R^{n\times n}$ is a real-analytic function of $\eta\in\S^{n-1}$, $n\ge 3$;
\item $A(\eta)=A^*(\eta)>0$ is self-adjoint and positive.
\end{itemize}
In general we can {\em not} assume that $A(\eta)$ is non-degenerate in the sense that
$\#\spec A(\eta)=n$ for all $\eta\in\S^{n-1}$ (as done for the two-dimensional case in 
\cite{RW07}). All basic examples show degeneracies in dimensions $n\ge3$.

\begin{expl} {\sl Isotropic media} 
\begin{equation}
  A(\eta)=\mu I + (\lambda+\mu) \eta\otimes\eta
\end{equation}
with Lam\'e constants $\lambda$ and $\mu$. The matrix $A(\eta)$ is positive as long as $\mu>0$
and $\lambda>-2\mu$. The eigenvectors of $A(\eta)$ are multiples of $\eta$ and $\eta^\perp$ 
and thus invariant under rotations of frequency space.
\end{expl}
\begin{expl}\label{expl2}{\sl Cubic media}
\begin{equation}\label{eq:1.3}
   A(\eta) = \begin{pmatrix} (\tau-\mu)\eta_1^2+\mu & (\lambda+\mu)\eta_1\eta_2 & \cdots &(\lambda+\mu)\eta_1\eta_n \\ (\lambda+\mu)\eta_1\eta_2 & (\tau-\mu)\eta_2^2+\mu & &\vdots \\ \vdots & & \ddots & \vdots\\ (\lambda+\mu)\eta_1\eta_n & \cdots & \cdots & (\tau-\mu)\eta_n^2+\mu   \end{pmatrix}
\end{equation}
described by parameters $\lambda$, $\mu$ and $\tau$. Later we will describe the 
assumptions made on these parameters and the resulting spectral properties of
the matrix function $A(\eta)$ more precise.  In the case of three space dimensions, the
matrix $A(\eta)$ is positive if and only if $\mu>0$, $\tau>0$ together with $ - 2\mu-\tau/2 < \lambda<\tau$. In three space dimensions this will
be one of our main examples.
\end{expl}

\begin{expl}
We can replace the constant $\tau$ on the diagonal by $\tau_1,\dots,\tau_n$ in 
\eqref{eq:1.3}. This yields so-called {\sl rhombic media}. The behaviour of
rhombic media is close to that of cubic media if the parameters are of similar
size, in general there will appear exceptional situations. See, e.g., \cite{Wirth:2007c} 
or \cite{Wirth:2009} for a discussion of this effect in two space dimensions.
\end{expl}

\begin{expl}{\sl Hexagonal media} are another particularly interesting case for 
three space dimensions. Since we want to come back to them later on we introduce 
the corresponding operator. It is given by
\begin{equation}\label{eq:1.4}
   A(\eta) = \mathbb D(\eta)^T \mathcal C \mathbb D(\eta),
\end{equation}
where $\mathcal C$ contains the $5$ structure constants $\tau_1$, $\tau_2$, 
$\lambda_1$, $\lambda_2$ and $\mu$ and $\mathbb D(\eta)$ is of a particular form,
\begin{equation}\label{eq:1.5}
\mathcal C= \begin{pmatrix} \tau_1&\lambda_1&\lambda_2\\
   \lambda_1&\tau_1&\lambda_2\\
   \lambda_2&\lambda_2&\tau_2\\
   &&&\mu\\
   &&&&\mu\\
   &&&&&\frac{\tau_1-\lambda_1}2 \end{pmatrix} ,
   \qquad
   \mathbb D(\eta) = \begin{pmatrix} \eta_1\\&\eta_2\\&&\eta_3\\&\eta_3&\eta_2\\\eta_3&&\eta_1\\\eta_2&\eta_1\end{pmatrix}.
\end{equation}
\end{expl}
%
Even the first (non-trivial anisotropic) example, the case of cubic media in 
three space dimensions, has degenerate directions in which $A(\eta)$ has double 
eigenvalues. Later on we will analyse this example in detail. 
\begin{df}
We call a direction $\eta\in\S^{n-1}$ (elastically) {\em non-degenerate} if
\begin{equation}
   \#\spec A(\eta) = n
\end{equation}
holds true for this direction $\eta$.
\end{df}

The set of non-degenerate directions is an open subset of $\S^{n-1}$. For 
non-degenerate directions the treatment of \cite{RW07} transfers almost immediately 
and gives a representation of solutions. We will sketch the results in Section~\ref{sec2}. 
In Section~\ref{sec3} we consider special degenerate directions and discuss the 
examples of cubic and hexagonal media. Dispersive estimates for solutions are given 
in Section~\ref{sec4}. In the neighbourhood of degenerate directions they are 
essentially based on estimates developed by Liess \cite{Lie04}, \cite{Lie06} for
the treatment of anisotropic acoustic equations.

\section{Treatment of non-degenerate directions}\label{sec2}
For the following we consider a simply connected open subset $\mathcal U$ of $\S^{n-1}$, 
where the symbol $A(\eta)$ has $n$ distinct (and real) eigenvalues. We denote these 
eigenvalues in ascending order as
\begin{equation}
 0< \varkappa_1(\eta)<\varkappa_2(\eta)<\cdots<\varkappa_n(\eta). 
\end{equation}
By analytic perturbation theory, see \cite{Kato:1980}, we know that these eigenvalues 
are real-analytic and that we find corresponding normalised eigenvectors
\begin{equation}
   r_1(\eta),\dots, r_n(\eta)\in C^\infty (\mathcal U,\S^{n-1})
\end{equation}
depending analytically on $\eta\in\mathcal U$.
Collecting them in the unitary matrix
\begin{align}
  &M(\eta) = \big( r_1(\eta) | r_2(\eta) | \cdots | r_n(\eta) \big),\\
  &M^*(\eta)M(\eta) = I = M(\eta)M^*(\eta),
\end{align}
we can diagonalise the matrix $A(\eta)$
\begin{align}
  &A(\eta)M(\eta) = M(\eta) \mathcal D(\eta),\\
  &\mathcal D(\eta) = \diag \big( \varkappa_1(\eta), \varkappa_2(\eta), \ldots ,\varkappa_n(\eta)\big).
\end{align}
In our treatment we will not make use of analyticity directly, instead our use
of perturbation theory will be based on \cite{JaWi} und \cite{Wirth:2008b} and 
uses only smooth dependence. This will be of interest for generalisations
later on. Therefore, whenever we use analyticity, we will explicitly state that.

We use $M(\eta)$ to reduce the thermo-elastic system to a system of first 
order. For this we denote by $\hat U$ and $\hat\theta$ the partial Fourier transforms 
of $U$ and $\theta$ with respect to the spatial variables and consider
\begin{equation}\label{eq:2.7}
  V = \begin{pmatrix}
  (\D_t+\mathcal D^{1/2}(\xi))M^*(\eta) \hat U\\
  (\D_t-\mathcal D^{1/2}(\xi))M^*(\eta) \hat U\\
 \hat \theta 
  \end{pmatrix} \in \C^{2n+1},
\end{equation}
as usual $\D_t=-\i\partial_t$ and $\eta=\xi/|\xi|$. Then $V$ satisfies a first order 
system of ordinary differential equations, which has an  apparantly simple structure. 
Straightforward calculation shows that
\begin{equation}
\D_t V = B(\xi) V
\end{equation}
holds true with coefficient matrix
\begin{equation}\label{eq:2.9}
  B(t,\xi) = \begin{pmatrix} 
  \omega_1(\xi) & & & & & & \i\gamma a_1(\xi)\\
  & \omega_2(\xi) & & & & & \i\gamma a_2(\xi)\\
  & & \ddots & & & & \vdots \\
  & & & -\omega_1(\xi) & & & \i\gamma a_1(\xi)\\
  & & & & -\omega_2(\xi) & & \i\gamma a_2(\xi)\\
  & & & & & \ddots & \vdots \\
 \frac{\i\gamma}2 a_1(\xi)& \frac{\i\gamma}2 a_2(\xi)&\cdots& \frac{\i\gamma}2 a_1(\xi)& \frac{\i\gamma}2 a_2(\xi)&\cdots & \i\kappa|\xi|^2
  \end{pmatrix},
\end{equation}
where $\omega_j(\xi)=\sqrt{\varkappa_j(\xi)}\in C^\infty(\mathcal U,\R_+)$ and
\begin{equation}
  a_j(\xi) = r_j(\eta)\cdot\xi.
\end{equation}
Following the conventions of \cite{RW07} we denote these functions 
$a_j(\xi)$ as the {\em coupling functions} of the thermo-elastic system 
associated to the elastic operator $A(\D)$. They play  a prominent r\^ole 
for the description of the time-asymptotic behaviour of solutions. This 
reflects the fact that they couple the homogeneous first order entries 
in $B(\xi)$ with the second order lower right corner entry. Note, that
\begin{equation}
   \sum_{j=1}^n a_j^2(\eta) = 1.
\end{equation}
Zeros of the coupling functions are of particular importance. Following 
\cite[Def.~1]{RW07} we define:

\begin{df}
  A non-degenerate direction $\eta\in\S^{n-1}$ is called
  \begin{itemize}
  \item {\em hyperbolic} if one of the coupling functions vanishes; more precisely,
  it is called {\em hyperbolic with respect to the eigenvalue $\varkappa_j(\eta)$} if $a_j(\eta)=0$;
  \item {\em parabolic} if all coupling functions are non-zero.
  \end{itemize}
\end{df}

In the anisotropic case the set of hyperbolic directions is 
(generically\footnote{If not, by analyticity it follows that one coupling 
  function vanishes on $\mathcal U$ and the system is therefore decoupled. 
  This case is reduced to the study of the lower dimensional blocks, one 
  is a hyperbolic system the other one a thermo-elastic system of lower 
  dimension. This is, e.g., the case for hexagonal media, see 
  Section~\ref{sec3.3}.}) a lower dimensional subset of $\S^{n-1}$. In order to 
decide whether a direction is hyperbolic or parabolic we can employ the following
proposition. We denote for a matrix $A$ and a vector $\eta$ by
\begin{equation}
   \mathcal Z(A,\eta) = \spann \{\, A^k\eta\,|\,k=0,1,\ldots\,\}
\end{equation}
the corresponding cyclic subspace, i.e. the span of the trajectory of $\eta$ 
under the action of the matrix $A$.

\begin{prop}\label{prop:2.1}
The following statements are equivalent:
\begin{enumerate} 
\item The cyclic subspace of $\eta$ has dimension $n-k$, i.e.,
  $\dim\mathcal Z(A(\eta),\eta) = n - k$.
\item Exactly $k$ of the coupling functions vanish in $\eta$.
\end{enumerate}
Hence, a  non-degenerate direction 
$\eta\in\S^{n-1}$ is parabolic if and only if  $\mathcal Z(A(\eta),\eta)=\R^n$ and therefore
\begin{equation}
  \det\big(\eta | A(\eta)\eta | \cdots | A^{n-1}(\eta)\eta \big) \ne 0.
\end{equation}
\end{prop}
\begin{proof}
If we represent $\eta$ in the eigenbasis of $A(\eta)$ we obtain
\begin{equation}
  \eta = a_1(\eta) r_1(\eta) + \cdots + a_n(\eta)r_n(\eta)
\end{equation}
and therefore
\begin{equation}
  A^\ell(\eta)\eta  = \varkappa_1^\ell(\eta) a_1(\eta) r_1(\eta)+ \cdots + \varkappa_n^\ell(\eta)a_n(\eta)r_n(\eta).
\end{equation}
If $k$ of the coupling functions vanish, then $A^{n-k}(\eta)\eta$ 
must be in the span of the $A^\ell(\eta)\eta$ with $\ell=0,1,\ldots, n-k-1$ 
and thus the cyclic subspace is at most  of dimension $n-k$. On the 
other hand, the first $n-k$ vectors in the trajectory are linearly independent 
since the corresponding matrix in the basis representation with respect to 
$a_1(\eta)r_1(\eta),\dots,a_n(\eta)r_n(\eta)$ is just the van der Monde 
matrix associated to the eigenvalues of $A(\eta)$ for non-vanishing 
coupling functions and therefore regular. 
\end{proof}

\subsection{On the characteristic polynomial of the full symbol} At first we collect 
some of the spectral properties of the matrix $B(\xi)$ which are directly related 
to the characteristic polynomial of $B(\xi)$.

\begin{prop}\label{prop22}
The following identies hold true:
\begin{align}
\trace B(\xi) &= \i\kappa|\xi|^2,\\
\det B(\xi) &= \i\kappa|\xi|^{2}\det A(\xi),\\
\det(\nu-B(\xi)) &= (\nu-\i\kappa|\xi|^2) \prod_{j=1}^n (\nu^2-\varkappa_j(\xi))-\nu\gamma^2\sum_{j=1}^n a_j^2(\xi) \prod_{k\ne j}(\nu^2-\varkappa_k(\xi)).\label{eq:CharPolB}
\end{align}
Furthermore, the matrix $B(\xi)$ has a purely real eigenvalue for $\xi\ne0$ if 
and only if the direction $\eta=\xi/|\xi|$ is hyperbolic. If it is $j$-hyperbolic, 
then $\pm\omega_j(\xi)\in\spec B(\xi)$.
\end{prop}

The proof of the last fact is fairly straightforward and consists of 
separating real and imaginary parts of the characteristic polynomial. 
Note that for all parabolic directions we can divide the characteristic 
polynomial by $\nu\prod_j (\nu^2-\varkappa_j(\xi))$ to obtain
\begin{equation}\label{eq:KelvinCharEq}
   1 = \frac{\i\kappa|\xi|^2}\nu + \gamma^2 \sum_{j=1}^n \frac{a_j^2(\xi)}{\nu^2-\varkappa_j(\xi)}.
\end{equation}
This formulation allows to consider the neighbourhoods of hyperbolic directions. 
Assume for this that the set of hyperbolic directions with respect to $\varkappa_j(\eta)$ 
\begin{equation}
   M_j = \{ \eta\in\mathcal U\,|\, a_j(\eta) = r_j(\eta)\cdot \eta=0\}
\end{equation}
is a regular submanifold of $\mathcal U$. If we consider the 
corresponding {\em hyperbolic} eigenvalues $\nu_j^\pm(\xi)$ of $B(\xi)$ in a 
neighbourhood of $M_j$, i.e. the eigenvalues which satisfy
\begin{equation}
  \lim_{\eta\to M_j} \nu_j^\pm(\xi)=\pm\omega_j(\xi)
\end{equation}
for fixed $|\xi|$,  equation \eqref{eq:KelvinCharEq} gives a precise 
description of the behaviour  of the imaginary part of these eigenvalues. 
The proof is a straightforward generalisation from \cite[Prop.~2.2]{RW07}.

\begin{prop}\label{prop:2.3}
  The non-tangential limit
 \begin{equation}
   \lim_{\eta\to M_j} \frac{a_j^2(\xi)}{\nu_j^\pm(\xi)^2-\varkappa_j(\xi)}
   = 1 \mp \frac{\i\kappa|\xi|^2}{\omega_j(\xi)} - \gamma^2\sum_{k\ne j} \frac{a_k^2(\xi)}{\varkappa_j(\xi)-\varkappa_k(\xi)}=\gamma^2(C_{\bar\eta}\mp\i D_{\bar\eta}|\xi|)
 \end{equation} 
  exists and is non-zero for all $\xi\ne0$. Furthermore,
  \begin{equation}
  \lim_{\eta \to M_j} \frac{\Im \nu_j^\pm(\xi)}{a_j^2(\eta)} = \frac{D_{\bar\eta} |\xi|^2}{2\omega_j(\bar\eta) (C_{\bar\eta}^2+|\xi|^2\D_{\bar\eta}^2)} > 0.
  \end{equation}
\end{prop}

\subsection{Asymptotic expansion of the eigenvalues as $|\xi|\to0$} We decompose $B(\xi)$
into homogeneous components $B(\xi)=B_1(\xi)+B_2(\xi)$ of degree $1$ and $2$, 
respectively. For sufficiently small $|\xi|$ we expect the eigenvalues of $B(\xi)$ 
to be close to the eigenvalues of $B_1(\xi)$. For parabolic directions the 
(non-zero) eigenvalues of $B_1(\eta)$ can be determined from the equation
\begin{equation}\label{eq:2.24}
   \frac1{\gamma^2} = \sum_{j=1}^n \frac{a_j^2(\eta)}{\tilde\nu^2-\varkappa_j(\eta)}, 
\end{equation}
which follows directly from \eqref{eq:KelvinCharEq} with $\kappa=0$. It 
can be solved (e.g. graphically, see Figure~\ref{fig1} for $n=3$) to 
obtain the distinct eigenvalues $0$, $\pm\tilde\nu_1(\eta)$, \dots, $\pm\tilde\nu_n(\eta)$ 
ordered as
\begin{equation}
   0 < \omega_1(\eta) < \tilde\nu_1(\eta) < \omega_2(\eta) < \tilde\nu_2(\eta) < \cdots < \omega_n(\eta) < \tilde\nu_n(\eta).
\end{equation}

\begin{figure}
\includegraphics[width=8cm]{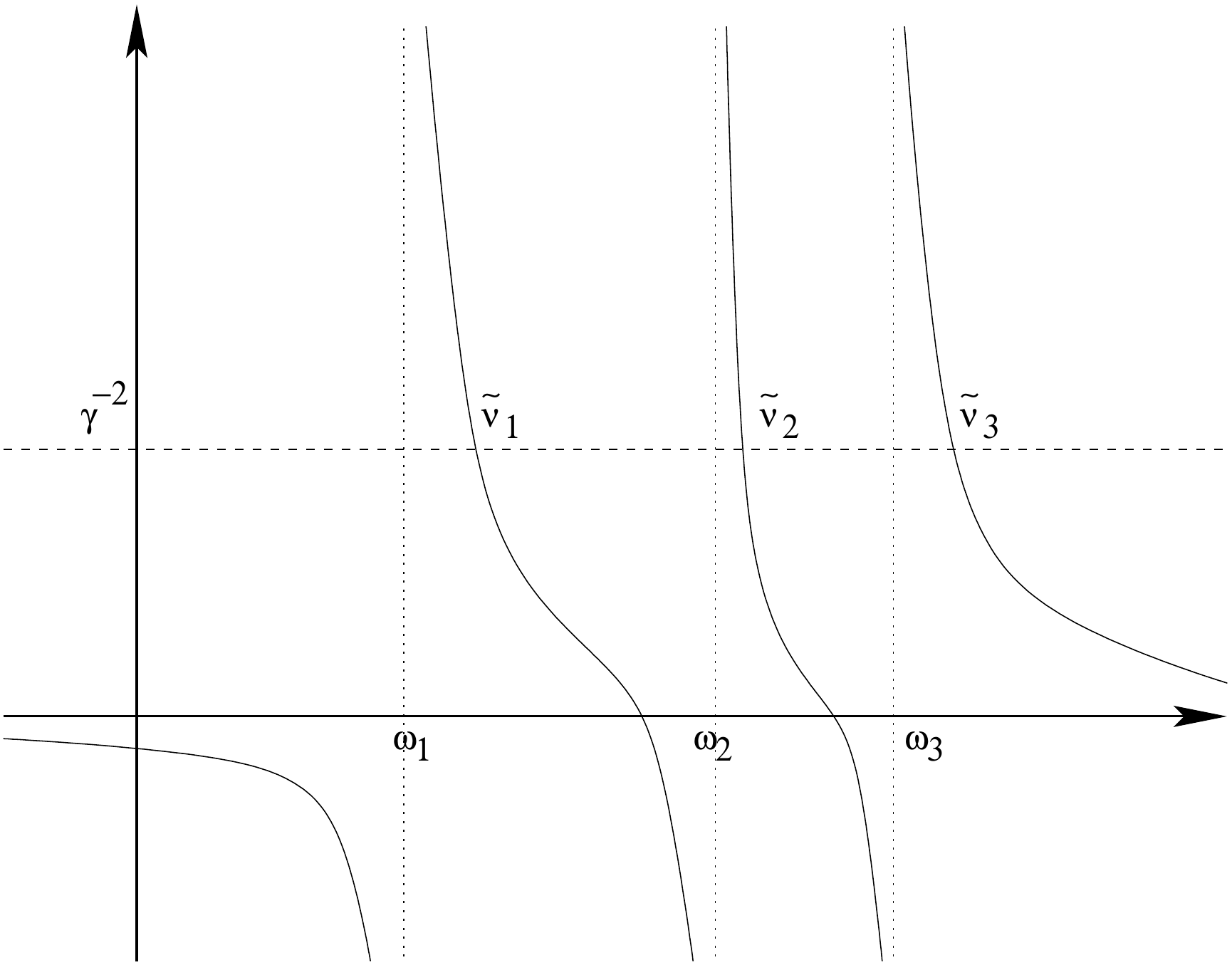}
\caption{Non-zero eigenvalues of $B_1(\xi)$ for parabolic directions.}\label{fig1}
\end{figure}

For hyperbolic directions a similar result holds true. In the case of 
hyperbolic directions  w.r.to $\varkappa_j(\eta)$ eigenvalues move to 
$\omega_j(\eta)$. According to the choice of the coupling constant 
$\gamma$ different cases occur:
\begin{enumerate}
\item if $\frac1{\gamma^2}$ is large then $\tilde\nu_j(\eta)=\omega_j(\eta)$, the other inequalities are unchanged;
\item if $\frac1{\gamma^2}$ is small then $\tilde\nu_{j-1}(\eta)=\omega_j(\eta)$ and the other inequalities remain true.
\end{enumerate}
The critical threshold between these two cases is
\begin{equation}\label{eq:2.27}
   \frac1{\gamma^2} = \sum_{k\ne j} \frac{a_k^2(\eta)}{\varkappa_j(\eta)-\varkappa_k(\eta)},
\end{equation}
where $B_1(\eta)$ has the {\em double} eigenvalue 
$\tilde\nu_{j-1}(\eta)=\omega_j(\eta)=\tilde\nu_j(\eta)$. Following the 
conventions from \cite{RW07} we define:

\begin{df} 
  We denote  a hyperbolic direction w.r.to $\varkappa_j(\eta)$ as 
$\gamma$-degenerate if \eqref{eq:2.27} holds true.
\end{df}

For the following treatment we exclude $\gamma$-degenerate hyperbolic 
directions and assume instead that for all hyperbolic directions in 
$\mathcal U$ condition \eqref{eq:2.27} is not satisfied for the 
corresponding index $j$. Then the following statement is apparent.

\begin{prop}\label{prop:2.4} Let $\eta$ be not $\gamma$-degenerate. Then
the matrix $B_1(\eta)$ has $2n+1$ distinct real eigenvalues $0, \pm\tilde\nu_1,\dots,\pm\tilde\nu_n$
for all $\eta\in\mathcal U$.
\end{prop} 

Proposition~\ref{prop:2.4} allows to apply the standard diagonalisation 
scheme (see \cite[Sec.~2.1]{JaWi}) to $B(\xi)=B_1(\xi)+B_2(\xi)$ as 
$\xi\to0$. Hence, eigenvalues, eigenprojections and all their derivatives 
have full asymptotic expansions as $\xi\to0$. The proof is almost 
identical to that from \cite[Prop.~2.5]{RW07} and is omitted.

\begin{prop}\label{prop:2.5}
 For all not $\gamma$-degenerate directions $\eta=\xi/|\xi|\in\mathcal U$ 
the eigenvalues and eigenprojections of $B(\xi)$ 
 have full asymptotic expansions as  $\xi\to0$. The main terms are given by
  \begin{subequations}
  \begin{align}
  \nu_0(\xi)&=\i\kappa|\xi|^2 b_0(\eta)+\mathcal O(|\xi|^3)\\
  \nu_j^\pm(\xi)&= \pm|\xi|\tilde\nu_j(\eta)+\i\kappa|\xi|^2 b_j(\eta)+\mathcal O(|\xi|^3)
  \end{align}
  \end{subequations}
  with
  \begin{subequations}\label{eq:b-def}
  \begin{align}
    b_0(\eta) &= \left(1+\gamma^2 \sum_{k=1}^n \frac{a_k^2(\eta)}{\varkappa_k(\eta)}\right)^{-1} >0\\
  \intertext{and}
    b_j(\eta) &= \left(1 + \gamma^2 \sum_{k=1}^n a_k^2(\eta) \frac{\tilde\nu_j^2(\eta)+\varkappa_k(\eta)}{(\tilde\nu_j^2(\eta)-\varkappa_k(\eta))^2}\right)^{-1} \ge 0.
    \end{align}
    \end{subequations}
    Furthermore,  $b_j(\eta)=0$ if and only if $\eta$ is hyperbolic with respect to the eigenvalue $\varkappa_j(\eta)$.
\end{prop}

\begin{rem}  Note, that $\trace B(\xi) = \i\kappa|\xi|^2$ implies 
\begin{equation}
   b_0(\eta)+2\sum_{j=1}^n b_j(\eta) =1.
\end{equation}
\end{rem}

Recall that by Proposition~\ref{prop22} eigenvalues of $B(\xi)$ can 
only be real along hyperbolic directions (and then they are exactly 
the 'trivial' real eigenvalues). In combination with the fact 
that eigenvalues of $B(\xi)$ are continuous in $\xi$ we obtain: 

\begin{cor}
  For all parabolic directions $\eta=\xi/|\xi|\in\mathcal U$ we have $\Im \nu_j^\pm (\xi) >0$.
  The same is true as long as $\eta$ is {\em not} hyperbolic w.r.to $\varkappa_j(\eta)$.
\end{cor}

\subsection{Asymptotic expansion of the eigenvalues as $|\xi|\to\infty$} In this case the two-step
procedure developed in \cite[Sec.~2.2]{JaWi}, \cite[Prop.~2.6]{RW07} applies in analogy. 
Essential assumption is the non-degeneracy of $A(\eta)$. We omit the proof and cite 
the corresponding result only.

\begin{prop}\label{prop:2.6}
For all non-degenerate directions the eigenvalues and eigenprojections of the matrix $B(\xi)$ 
have full asymptotic expansions as $|\xi|\to\infty$. The first terms are given by
\begin{subequations}
\begin{align}
  \nu_0(\xi) &= \i\kappa |\xi|^2 - \frac{\i\gamma}{\kappa} +\mathcal O(|\xi|^{-1}),\\
  \nu_j^\pm(\xi) &= \pm|\xi|\omega_j(\eta) + \frac{\i\gamma^2}{2\kappa} a_j^2(\eta) + \mathcal O(|\xi|^{-1}).
\end{align}
\end{subequations}
\end{prop}

\begin{rem} Despite the fact that we used the same notation for the eigenvalues 
as $\xi\to0$ and $|\xi|\to\infty$, we do not claim that they are indeed the same 
functions of $\xi$. This is only true for hyperbolic eigenvalues near hyperbolic 
directions, in general there might be multiplicities in between and there might 
be no consistent notation for these functions.
\end{rem}

\begin{cor}\label{cor:2.8}
   For all parabolic directions $\eta=\xi/|\xi|$ the eigenvalues of $B(\xi)$ 
satisfy $\Im \nu(\eta) \ge C_\eta > 0$ for $|\xi|\ge c$. The same is true for 
parabolic eigenvalues in hyperbolic directions.
\end{cor}

\begin{rem}
In particular,
we see by the asymptotic expansions that the eigenvalues of $B(\xi)$ are simple 
for large and also for small values of $|\xi|$. Furthermore, we see that the 
hyperbolic eigenvalues are always separeted (i.e. if multiplicities occur in 
hyperbolic directions, they involve only parabolic eigenvalues).
\end{rem}

\subsection{Behaviour of the imaginary part} The asymptotic expansions of 
Propositions~\ref{prop:2.5} and~\ref{prop:2.6} allow to draw conclusions for 
the behaviour of the imaginary part. We collect them for later use. The first 
result is apparent.

\begin{prop}\label{prop:2.9}
   On any compact set of parabolic directions we have the uniform estimates
   \begin{align}
    & \Im\nu_j^{(\pm)}(\xi) \ge C_\epsilon\qquad &&\text{for all $|\xi|\ge \epsilon$,}\\
     & \Im\nu_j^{(\pm)}(\xi) \sim b_j(\eta)|\xi|^2 \qquad &&\text{for all $|\xi|\le \epsilon$}
   \end{align}
   for all eigenvalues of $B(\xi)$ and arbitrary $\epsilon>0$.
\end{prop}

The next statement is concerned with a tubular neighbourhood of a compact 
subset of a regular submanifold $M_j$ of hyperbolic eigenvalues w.r.to 
$\varkappa_j(\eta)$. It is only of interest how the corresponding hyperbolic 
eigenvalues $\nu_j^\pm(\xi)$ behave, the others still satisfy Proposition~\ref{prop:2.9}.

\begin{prop} 
Uniformly on any tubular neighbourhood of a compact subset of $M_j$ of 
non-$\gamma$-degenerate directions the corresponding hyperbolic eigenvalues 
$\nu_j^\pm(\xi)$ satisfy the estimates 
  \begin{align}
    & \Im\nu_j^\pm(\xi) \sim a_j^2(\eta) \qquad &&\text{for all $|\xi|\ge \epsilon$,}\\
     & \Im\nu_j(\xi) \sim b_j(\eta)|\xi|^2 \qquad &&\text{for all $|\xi|\le \epsilon$.}
   \end{align}
\end{prop}
\begin{proof} By Proposition~\ref{prop:2.3} we know that
\begin{equation}
   \Im \nu_j^{\pm} (\xi) = a_j^2(\xi) K(\xi) 
\end{equation}
for some function $K(\xi)$. Our aim is to estimate $K(\xi)$. The left hand of 
this formula has a full asymptotic expansion as $|\xi|\to0$ and $|\xi|\to\infty$. 
Therefore, also the right hand side has one and
it follows that
\begin{subequations}
\begin{align}
K(\xi) &= \frac{\gamma^2}{2\kappa} + \mathcal O(|\xi|^{-1}),\qquad &&|\xi|\to\infty,\\
K(\xi) &= \kappa|\xi|^2\frac{b_j(\eta)}{a_j^2(\eta)}+ \mathcal O(|\xi|^{-3}),\qquad &&|\xi|\to0.
\end{align}
\end{subequations} 
Thus, the desired estimate follows by a compactness argument as soon as we 
have a uniform lower/upper bound for $b_j(\eta) / a_j^2(\eta)$. The representation 
of $b_j(\eta)$ in Proposition~\ref{prop:2.5} in combination with \eqref{eq:2.24} implies
\begin{align}
  \lim_{\eta\to M_j} \frac{a_j^2(\eta)}{b_j(\eta)} = & 
  \lim_{\eta\to M_j} 
  \gamma^2(\tilde\nu_j^2+\varkappa_j(\eta))  \frac{a_j^4(\eta)}{(\tilde\nu_j^2-\varkappa_j(\eta))^2}\notag\\
   &+ \lim_{\eta\to M_j} a_j^2(\eta)\left(1 + \gamma^2\sum_{k\ne j} a_k^2(\eta) \frac{\tilde\nu_j^2+\varkappa_k(\eta)}{(\tilde\nu_j^2-\varkappa_k(\eta))^2} 
  \right)\notag\\
  =&2\gamma^2\varkappa_j(\bar\eta) \left(1-\gamma^2\sum_{j\ne k} \frac{a_k^2(\bar\eta)}{\varkappa_j(\bar\eta)-\varkappa_k(\bar\eta)} \right)^2,
\end{align}
which is clearly bounded and (uniformly) positive on any compact subset of 
$M_j$ (where we have to use that $\bar\eta\in M_j$ is not $\gamma$-degenerate).
\end{proof}

\subsection{Conclusions} We will draw several conclusions what we have 
obtained so far and what we still have to consider in the remaining part 
of this treatise.

\subsubsection{Cubic media in 3D} If we consider the special case of cubic 
media in three space dimensions degenerate directions are given by 
$\bar\eta=(\bar\eta_1,\bar\eta_2,\bar\eta_3)^T$ with $\bar\eta_1^2=\bar\eta_2^2=\bar\eta_3^2$  (eight directions, corresponding to the corners of a cube) or $\bar\eta_i^2=1$ for some $i$ (six directions, 
corresponding to its faces). This can be calculated directly, corresponding 
eigenspaces are $\spann \{\bar\eta\}$ and $\bar\eta^\perp=\{\xi\in\R^n\;|\; \bar\eta\cdot\xi=0\}$, 
or concluded by the cubic symmetry\footnote{$A(\xi)$ is invariant 
  under the hexaeder group, i.e. the symmetry group of a cube. Thus, 
  eigenspaces must be transferred in an appropriate way, which implies 
  that symmetries of order 3 or 4 can only be realised by higher 
  dimensional eigenspaces.} of $A(\xi)$ in this particular case. See Figure~\ref{fig2}.

To obtain the hyperbolic directions we apply Proposition~\ref{prop:2.1} and 
look for the action of $\eta$ under $A(\eta)$. We obtain that 
\begin{enumerate}
\item a direction $\eta$ is hyperbolic if and only if
\begin{equation}
\det(\eta|A(\eta)\eta|A^2(\eta)\eta)= (\tau-\lambda-2\mu)^3 \eta_1\eta_2\eta_3(\eta_1^2-\eta_2^2)(\eta_1^2-\eta_3^2)(\eta_2^2-\eta_3^2)=0,
\end{equation}
thus the set of hyperbolic directions is the union of nine great circles on $\S^2$;
\item $\eta||A(\eta)\eta$ for all 26 intersection points of these great 
  circles, 14 of them are excluded as 
being degenerate. 
\end{enumerate}
Except for these 14 points on $\S^2$ we obtained an almost complete description 
of the spectrum of $B(\xi)$. We know full asymptotic expansions of eigenvalues 
for small and large frequencies $|\xi|$, estimates for the imaginary part of 
them and similar statements for eigenprojections. This information allows to 
draw conclusions on the large time behaviour of solutions, e.g. energy and 
dispersive estimates. This can be done similar to the treatment of \cite{RW07}, 
see Section~\ref{sec4}. The remaining degenerate directions appear in two types, 
which can be interchanged by the action of the symmetry group. The study of 
these degenerate directions is what is left open so far and will be the main 
point of Section~\ref{sec3}.
\begin{figure}
\includegraphics[width=6cm]{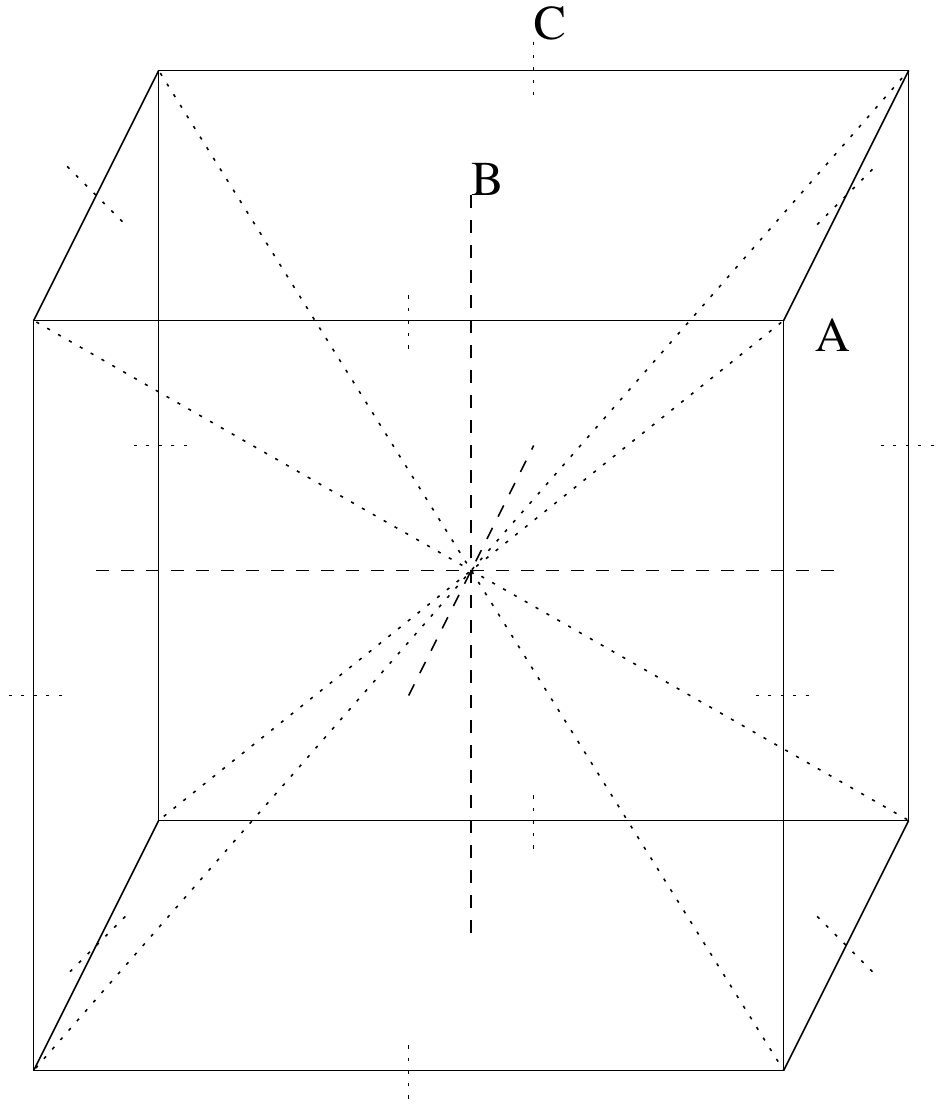}
\caption{Degenerate points for cubic media correspond to symmetries of a cube. 
Corner points A are conic singularities, midpoints of faces B uniplanar 
singularities of $\spec A(\eta)$. The midpoints of edges C are 
non-degenerate, but hyperbolic with respect to two different eigenvalues.}\label{fig2}
\end{figure}

\subsubsection{Isotropic media}
If we consider the special case of isotropic media, $A(\eta)=\mu I+(\lambda+\mu)\eta\otimes\eta$, we 
see that $\spec A(\eta)=\{\mu,\lambda+2\mu\}$ and corresponding eigenspaces are 
$\mathrm{span}\,\{\eta\}$ (corresponding to $\lambda+\mu$) and $\eta^\perp$ 
(corresponding to $\mu$). All directions are (elastically) degenerate. However, 
we still find locally smooth systems of eigenvectors. All directions are 
hyperbolic and the hyperbolic eigenvalue $\mu$ has multiplicity $n-1$. Therefore
the system $\D_tV=B(\xi)V$ decouples into a diagonal part of size $2n-2$ and 
a full $3\times3$ block and is given after a rearrangement of the entries as
\begin{equation}
  B(t,\xi) = \begin{pmatrix} 
 \sqrt \mu|\xi| & & & & & & & \\
  & \ddots & & & & & &\\
  & & -\sqrt\mu|\xi| & & & & &\\
  & & & \ddots & & & &\\
  & & & & \sqrt{\lambda+2\mu}|\xi| & & \i\gamma|\xi| \\
    & & &  & & -\sqrt{\lambda+2\mu}|\xi| & \i\gamma|\xi| \\
  & & &  & -\frac{\i\gamma}2|\xi|& -\frac{\i\gamma}2|\xi| & \i\kappa|\xi|^2
  \end{pmatrix}.
\end{equation}
This block structure corresponds to the Helmholtz decomposition of vector 
fields applied to the elastic displacement. If $\nabla\cdot U(t,\cdot)=0$ 
the lower block cancels and we obtain wave equations
with speed $\sqrt\mu$ for the components of $U$. Otherwise, if we cancel 
the upper block we obtain the $3\times3$ system corresponding to one-dimensional 
thermo-elasticity with its well-known properties.

\subsubsection{One-dimensional thermo-elasticity}\label{ssec:1d} For completeness 
we mention some results on the one-dimensional system
\begin{subequations}
\begin{align}
 u_{tt} - \tau^2 u_{xx} +\gamma \theta_x &=0,\\
 \theta_t-\kappa\theta_{xx} + \gamma u_{tx}&=0.
\end{align}
\end{subequations}
We assume $\gamma,\kappa,\tau>0$. Following our strategy we can rewrite this 
problem as first order system. The corresponding symbol
$B(\xi)$ is given by
\begin{equation}\label{eq:B-1d}
  B(\xi) = \begin{pmatrix} \tau \xi &&\i\gamma\xi\\&-\tau\xi&\i\gamma\xi\\
  -\frac\i2\gamma\xi&-\frac\i2\gamma\xi&\i\kappa\xi^2\end{pmatrix}.
\end{equation}
Its eigenvalues satisfy asymptotic expansions for $\xi\to0$ and 
$\xi\to\pm\infty$. Propositions~\ref{prop:2.5} and~\ref{prop:2.6} 
apply with $\tilde\nu^\pm=\pm\sqrt{\tau^2+\gamma^2}$ and 
\begin{equation}
b_0 = \frac{\tau^2}{\tau^2+\gamma^2},\qquad b_1=\frac12\frac{\gamma^2}{\tau^2+\gamma^2}.
\end{equation}
Therefore, by Proposition~\ref{prop:2.5}
\begin{subequations}\label{eq:2.43}
\begin{align}
  \nu_0(\xi) &= \i \frac{\kappa\tau^2}{\tau^2+\gamma^2}\xi^2 + \mathcal O(\xi^3),\\
  \nu_1^\pm(\xi) &= \pm\sqrt{\tau^2+\gamma^2} \xi + \i \frac{\kappa\gamma^2}{2(\tau^2+\gamma^2)} \xi^2+\mathcal O(\xi^3),
\end{align}
as $\xi\to0$ and by Proposition~\ref{prop:2.6}
\begin{align}
  \nu_0(\xi) &= \i \kappa \xi^2 -\i\frac{\gamma}{\kappa} + \mathcal O(\xi^{-1}),\\
  \nu_1^\pm(\xi) &= \pm\tau \xi + \i \frac{\gamma^2}{2\kappa} +\mathcal O(\xi^{-1}),
\end{align}
\end{subequations}
as $\xi\to\infty$. The essential information for large time estimates 
is given by the behaviour of the imaginary part. It follows that 
$\Im \nu(\xi) > C_\epsilon$ for $|\xi|\ge \epsilon$ for certain constants and 
\begin{equation}\label{eq:2.44}
  \Im \nu_0(\xi) \sim \frac{\kappa\tau^2}{\tau^2+\gamma^2}\xi^2,\qquad 
  \Im \nu_1^\pm(\xi) \sim  \frac{\kappa\gamma^2}{2(\tau^2+\gamma^2)} \xi^2,\qquad \xi\to0.
\end{equation}

\subsubsection{Hexagonal media in 3D}\label{sec:2:hex}
For hexagonal media in three space dimensions
the situation is (surprisingly) simpler than for cubic media. The elastic 
operator defined by \eqref{eq:1.4}--\eqref{eq:1.5} is invariant under rotation 
around the $x_3$-axis (taking into account a corresponding rotation of the
reference frame for vectors) and therefore it suffices to understand its 
cross sections in the $x_1$--$x_2$ plane. We will sketch some of the properties
of the corresponding symbol $A(\eta)$.

Following Proposition~\ref{prop:2.1} we obtain
\begin{enumerate}
\item that
  \begin{equation}
    \det (\eta | A(\eta)\eta | A^2(\eta)\eta )=0,
  \end{equation}
  such that {\em all} directions $\eta\in\mathbb S^2$ are hyperbolic. The 
  corresponding eigenspace is (generically) given by multiples of
  $(\eta_2,-\eta_1,0)$ such that the hyperbolic eigenvalue is
  \begin{equation}\label{eq:hexa-hyp-eig}
    \frac{\tau_1-\lambda_1}2 (\eta_1^2+\eta_2^2) + \mu \eta_3^2.
  \end{equation}
\item 
  It remains to look for directions with two hyperbolic eigenvalues. They
  satisfy $\eta||A(\eta)$. This is true, if $\eta_3=0$ or if $\eta_1=\eta_2=0$
  or if 
  \begin{equation}\label{eq:hexa-deg-circ}
    \eta_3^2 = \frac{\lambda_2 + 
      2 \mu - \tau_1} {2\lambda_2 + 
      4 \mu + \tau_1-\tau_2}, 
  \end{equation}
  provided the latter expression is non-negative. Except in the limiting case
  $\tau_1=\lambda_2+2\mu$, the coupling functions vanish to first order along
  the corresponding circle. If $\tau_1=\tau_2=\lambda_2+2\mu$ all directions are  hyperbolic with two hyperbolic eigenvalues and if  
  $\tau_1=\lambda_2+2\mu\ne \tau_2$ coupling functions vanish to third order.
\item 
  The matrices $A(\eta)$ are invariant under rotation. Introducing spherical
  coordinates on $\mathbb S^2$ 
  \begin{equation}
    \eta = \begin{pmatrix}1\\0\\0\end{pmatrix}\cos\phi\cos\psi
    + \begin{pmatrix}0\\1\\0\end{pmatrix}\sin\phi\cos\psi
    + \begin{pmatrix}0\\0\\1\end{pmatrix}\sin\psi
  \end{equation}
  and using a corresponding (moving) basis for
  vectors given by
  \begin{equation}
    \frac{\pm1}{\sqrt{\eta_1^2+\eta_2^2}} \begin{pmatrix}\eta_2\\-\eta_1\\0\end{pmatrix},\qquad \eta,\qquad  \frac{\pm\eta_3}{\sqrt{1-\eta_3^2}} \begin{pmatrix}\eta_1\\\eta_2\\-\frac{\eta_1^2+\eta_2^2}{\eta_3}\end{pmatrix}
  \end{equation}
  (sign chosen to make them smoothly dependent on $\eta\ne\pm(0,0,1)^\top$)
  decomposes $A(\eta)$ into (1,2)-block-diagonal structure 
  (independent of the co-ordinate $\phi$). The scalar block corresponds to
  the eigenvalue \eqref{eq:hexa-hyp-eig}, while the $2\times2$ block has
  trace $\mu+\tau_1\cos^2\psi + \tau_2\sin^2\psi$ and determinant
  $\mu\tau_1\cos^4\psi+\mu\tau_2\sin^4\psi+\frac{\tau_1\tau_2-2\lambda_2-\lambda_2^2}4 \sin^22\psi$.
\end{enumerate}
If $(\tau_1-\mu)(\tau_2-\mu)\ne0$, the $2\times2$ block has distinct 
eigenvalues for all $\psi$ and therefore the only degenerate directions are 
directions where this block has  \eqref{eq:hexa-hyp-eig} as one of its 
eigenvalues. This happens if and only if the right hand side of 
\eqref{eq:hexa-deg-circ} is non-negative and on the circle defined by that 
equation.

Thus, the previously developed theory is applicable for all directions except
the degenerate ones $\eta_1=\eta_2=0$ or \eqref{eq:hexa-deg-circ}. The always 
existent hyperbolic eigenvalue \eqref{eq:hexa-hyp-eig} leads to a decoupling 
of the thermo-elastic system into two scalar blocks and a (at least formally) 
2D thermo-elastic system. 

Due to rotational invariance, it suffices to treat the cut $\eta_1=0$
for handling of degenerate directions. This will be sketched later.

\section{Some special degenerate directions}\label{sec3}
We want to study neighbourhoods of degenerate directions for some 
particular cases. To study degenerate directions in full generality 
is beyond the scope of this paper. We relate our approach to the type 
of singularity of the corresponding {\em Fresnel surface} 
\begin{equation}
    \mathcal S=\{ \xi\in\R^n \, | \, 1 \in\spec A(\xi) \}.
\end{equation}
This surface is in general $n$-sheeted and for all non-degenerate directions these
sheets are given by
\begin{equation}
 \mathcal S_j= \{ \xi\in\mathbb R^n\;\text{non-deg.} \,|\, \omega_j(\xi)=1\} =
 \{ \omega_j^{-1}(\eta)\eta\,|\,\eta\in\S^{n-1}\;\text{non-degenerate}\, \}, 
\end{equation}
while in degenerate points the surface is self-intersecting. 
For the importance of these surfaces in elasticity theory and some 
interesting properties of them we refer to Duff \cite{Duf60} or the investigations 
from Musgrave \cite{Mus54}, \cite{Mus54a} and Miller-Musgrave \cite{Mus57}.

\begin{figure}
\includegraphics[width=7cm]{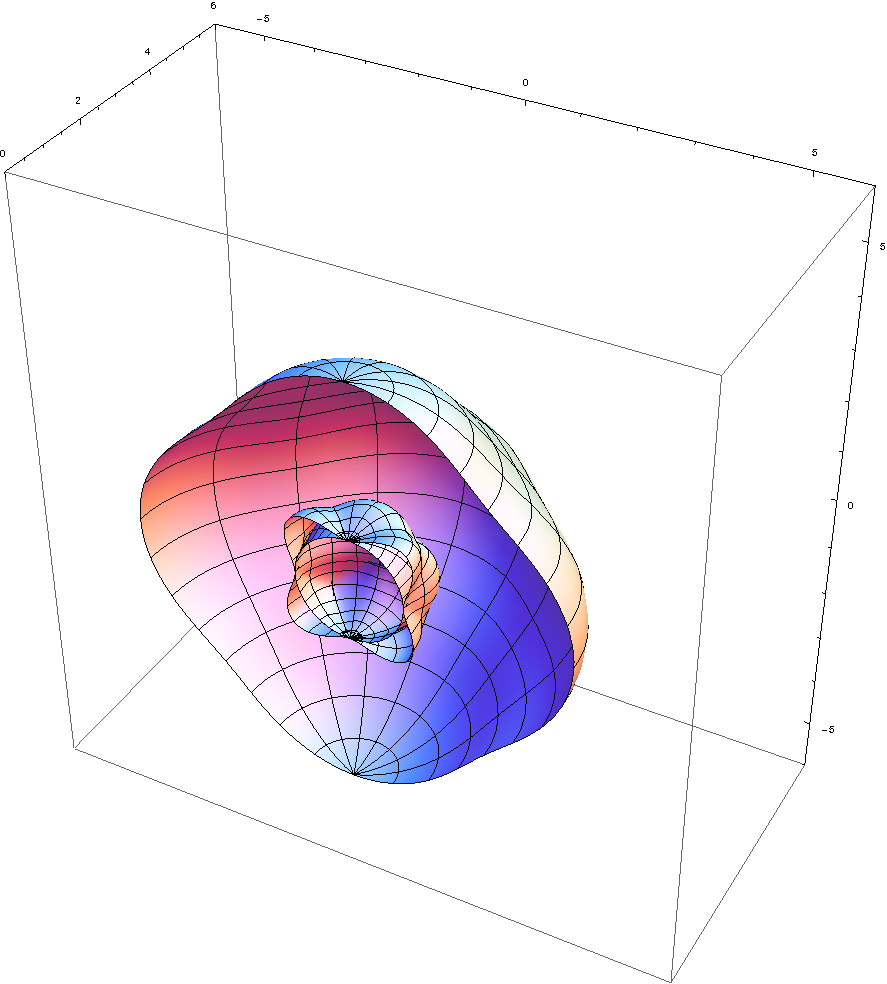}
\includegraphics[width=7cm]{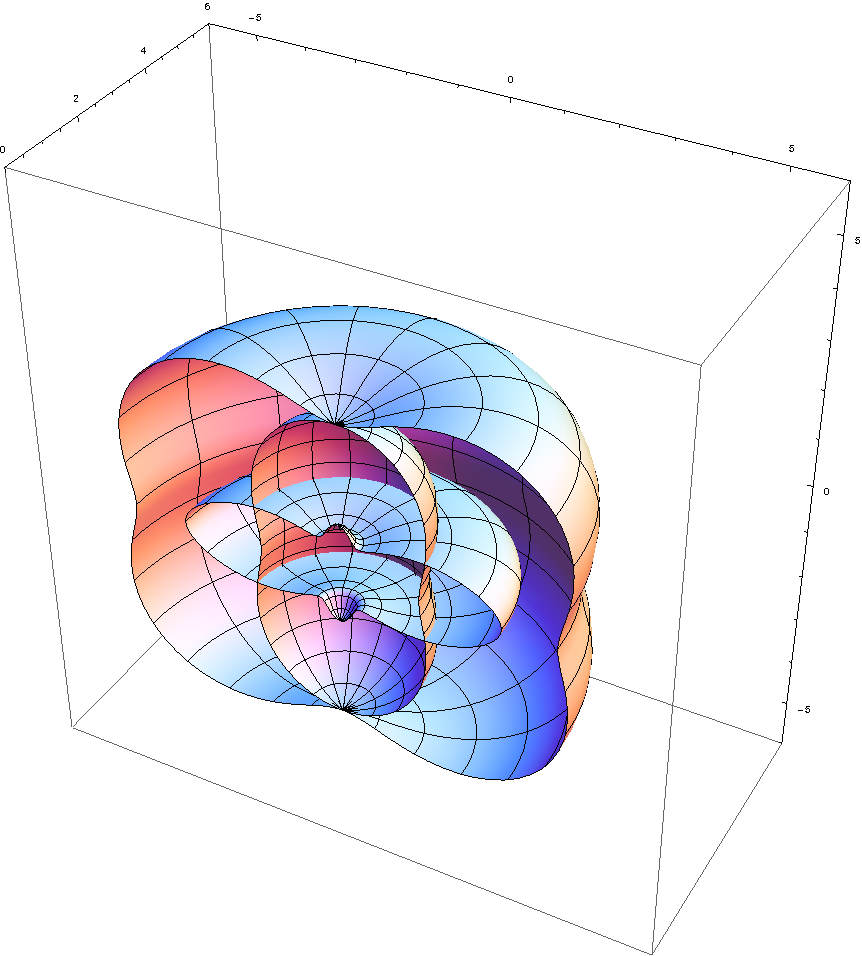}
\caption{A cut through the Fresnel surfaces for examples of a cubic and a
hexagonal medium. The material parameter are $\lambda=1$, $\tau=4$ and $\mu=1$ 
for the picture on the left (cubic) and $\lambda_1=1$, $\lambda_2=\frac15$, $\tau_1=4$, $\tau_2=1$ and $\mu=3$ for the picture on the right (hexagonal).}
\end{figure}

We remark only one of the general properties of $\mathcal S$ here. If $A(\xi)$ 
is polynomial in $\xi$ then the surface $\mathcal S$ is algebraic of 
degree $2n$ and therefore any straight line intersecting $\mathcal S$ has 
at most $2n$ intersection points with $\mathcal S$. In particular, if 
the inner sheet $\mathcal S_n$ does not touch any of the the outer sheets, it 
has to be strictly convex. 

\subsection{General strategy} If we investigate isolated degenerate directions or 
regular manifolds of degenerate directions of codimension greater than one we 
are faced with two major obstacles. Generically, eigenvectors of $A(\eta)$ can 
not be chosen continuously in a neighbourhood of the degenerate direction and 
therefore a reformulation as system of first order as in \eqref{eq:2.7} is problematic. 
This problem is related to higher-dimensional perturbation theory of matrices. 
It is well-known that in the one-dimensional situation eigenspaces are continuous 
(see, e.g., the book of Kato, \cite{Kato:1980}) and it can be resolved by 
introducing polar co-ordinates / normal co-ordinates around the degenerate directions 
and a system related to \eqref{eq:2.7} can be formulated on a corresponding blown-up 
space (see, e.g., \eqref{eq:3.10} below). Second obstacle are the multiplicities 
itself. Eigenvalues and eigenvectors of the constructed system of first order do 
not possess asymptotic expansions in powers of $|\xi|$ as $|\xi|$ tends to $0$ or 
$\infty$. However, especially in the three-dimensional setting we can write 
full asymptotic expansions in the distance to the degeneracy uniform in 
all remaining co-ordinates. 

We will discuss the application of this general strategy in detail for conic 
singularities of the Fresnel surface appearing for the case of cubic media and 
give the corresponding results for uniplanar singularities afterwards. Finally 
we will consider hexagonal media and show that they are much simpler in 
their analytical structure.

\begin{figure}
~\hfill
\includegraphics[width=5cm]{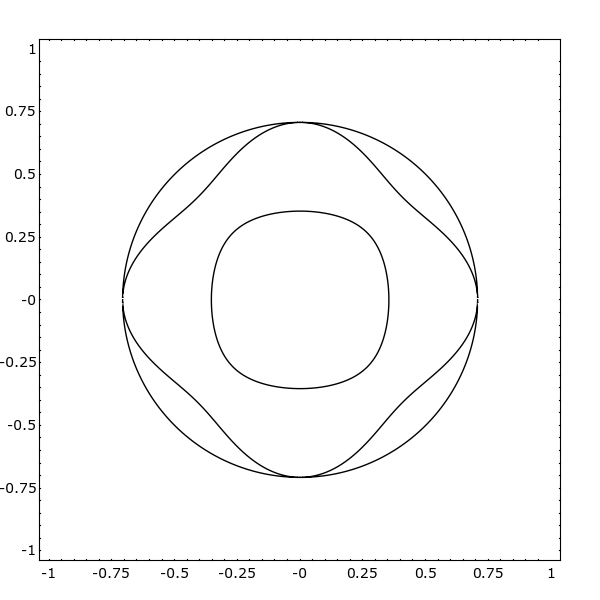}\hfill\includegraphics[width=7.07cm]{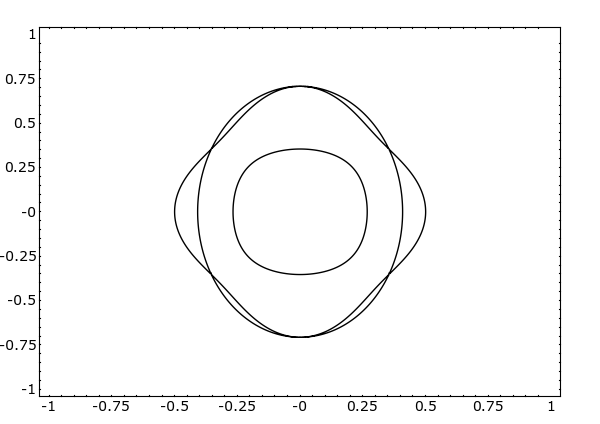}
\hfill~
\caption{Cuts of the Fresnel surface for cubic media; on the left hand side in the plane $\eta_3=0$, on the  right for $\eta_1=\eta_2$. The parameters are chosen as $\tau=8$, $\lambda=2$ and $\mu=2$.}
\end{figure}

\subsection{Cubic media, conic singularities} The Fresnel surface for cubic media 
has eight conic singularities which are related by the symmetries of the 
underlying medium. It suffices to consider one of them and we choose 
$\bar\eta=\frac1{\sqrt3}(1,1,1)^T\in\S^2$. 
Near this direction we introduce polar co-ordinates $(\epsilon,\phi)$ on 
the sphere $\S^2$ by
\begin{equation}\label{eq:3.3}
   \eta = \begin{pmatrix}\eta_1\\\eta_2\\\eta_3\end{pmatrix} = 
   \sqrt{1-\epsilon^2}\frac1{\sqrt3} \begin{pmatrix}1\\1\\1\end{pmatrix}
   + \epsilon \frac1{\sqrt6} \begin{pmatrix}-1\\-1\\2\end{pmatrix} \cos\phi + \epsilon \frac1{\sqrt2} \begin{pmatrix}1\\-1\\0\end{pmatrix} \sin\phi. 
\end{equation}
They allow to blow up the singularity by looking at $[0,\infty)\times\S^1$ 
instead of $\R^2$ as local model of $\S^2$.
In order to simplify notation, we apply a diagonaliser $\tilde M$ of 
$A(\bar\eta)$ to our matrices. For this we choose
the unitary matrix 
\begin{equation}
  \tilde M = \frac1{\sqrt6}\begin{pmatrix} \sqrt2 & -1 & \sqrt3\\\sqrt2&-1&-\sqrt3\\\sqrt2&2&0\end{pmatrix}  
\end{equation}
(corresponding to the vectors chosen already in \eqref{eq:3.3}).
The matrix $\tilde M^{-1}A(\eta) \tilde M$ has a full asymptotic 
expansion as $\epsilon\to0$ and can be written as
\begin{equation}
   \tilde  M^{-1}A(\epsilon,\phi)\tilde M = A_0 +  \epsilon A_1(\phi) 
    + \mathcal O(\epsilon^2),\qquad \epsilon\to0
\end{equation}
with matrices 
\begin{subequations}
\begin{align}
  A_0 & = \diag\left(\frac{\tau+2\lambda+4\mu}3, \frac{\tau+\mu-\lambda}3, \frac{\tau+\mu-\lambda}3\right),\\
  A_1(\phi)&=\frac{2\tau-\mu+\lambda}3 \begin{pmatrix} 
 &\cos\phi&\sin\phi\\\cos\phi&&\\\sin\phi    \end{pmatrix}\notag\\&\qquad\qquad
 + \frac{\sqrt2(-\tau+2\mu+\lambda)}3 \begin{pmatrix} 0 && \\&-\cos\phi&\sin\phi\\&\sin\phi&\cos\phi\end{pmatrix}. 
\end{align}
\end{subequations}
Now we can apply the block-diagonalisation procedure (again 
following \cite[Sec.~2.2]{JaWi}) to obtain the behaviour of 
eigenvalues and eigenprojections of $\tilde M^{-1}A(\epsilon,\phi)\tilde M$
as $\epsilon\to0$ for all $\phi$. We restrict consideration to the 
case where $\lambda+\mu\ne0$, such that 
$A_0$ has two different eigenvalues. 
\begin{prop}\label{prop:3.1}
The eigenvalues $\varkappa_j(\epsilon,\phi)$ and the corresponding 
eigenprojections of $A(\epsilon,\phi)$ have uniformly in $\phi$ full 
asymptotic expansions as $\epsilon\to0$. The main terms are given by
\begin{subequations}
\begin{align}
  \varkappa_1(\epsilon,\phi) &=  \frac{\tau+2\lambda+4\mu}3 +\mathcal O(\epsilon^2) ,\\
  \varkappa_2(\epsilon,\phi) &=  \frac{\tau+\mu-\lambda}3 +\frac{\sqrt2(-\tau+2\mu+\lambda)}3\epsilon  +\mathcal O(\epsilon^2) ,\\
  \varkappa_3(\epsilon,\phi) &=  \frac{\tau+\mu-\lambda}3 -\frac{\sqrt2(-\tau+2\mu+\lambda)}3\epsilon+\mathcal O(\epsilon^2) .
\end{align}
\end{subequations}
\end{prop}
\begin{rem}
The exceptional case when $\tau=\lambda+2\mu$ corresponds to isotropic 
media and is therefore excluded. In all other cases the two sheets 
$\omega_2(\eta)=\sqrt{\varkappa_2(\eta)}$ and 
$\omega_3(\eta)=\sqrt{\varkappa_3(\eta)}$ form a double cone on the 
Fresnel surface $S$. Hence, the statement explains the notion of 
conical singularity.   Note, that the linear terms are independent of $\phi$ and
therefore the cone is approximately a spherical cone near the conic point. 
\end{rem}
\begin{proof}
We will only shortly review the main steps. First we 
(1,2)-block-diagonalise $\tilde M^{-1}A(\epsilon,\phi)\tilde M$ 
(modulo $\mathcal O(\epsilon^N)$ for any $N$ we like). The 
diagonaliser we are going to construct has the form 
$I+\epsilon N_1^{(1)}(\phi)+\cdots +\epsilon^{N-1}N_1^{(N-1)}(\phi)$ 
and as in \cite[Sec.~2.2]{JaWi} its terms are given by 
recursion formulae. For $N_1^{(1)}$ we divide the
off-(block-)diagonal terms of $A_1$ by the difference of the 
corresponding diagonal entries of $A_0$. 
This gives as first term
\begin{equation}
   N_1^{(1)}(\phi) = \frac{2\tau-\mu+\lambda}{3(\lambda+\mu)} \begin{pmatrix} &\cos\phi & \sin\phi\\-\cos\phi\\-\sin\phi \end{pmatrix}
\end{equation}
and allows to cancel the off-(block-)diagonal entries of $A_1$. 
We skip the further construction and move to the next step.  Since 
the lower $2\times2$ block of $\bdiag_{1,2}A_1$ has distinct eigenvalues
(namely $\pm 1$) we can now perform a diagonalisation scheme in 
the subspace corresponding to this block (modulo $\mathcal O(\epsilon^N)$). 
Again we restrict ourselves to the main terms. A unitary
diagonaliser of the $2\times2$-block can be chosen as the unitary matrix
\begin{equation} 
 \tilde M_2(\phi) = 
 \begin{pmatrix}1  \\&
 \sin\frac\phi2 &\cos\frac\phi2 \\ &\cos\frac\phi2 &- \sin\frac\phi2
\end{pmatrix}.
\end{equation}
After transforming with that matrix we apply the recursive scheme 
to diagonalise further. Note that  after 
applying $\tilde M_2(\phi)$ the matrix is diagonal modulo 
$\mathcal O(\epsilon^2)$ and therefore, 
$\tilde M(I+\epsilon N_1^{(1)}(\phi))\tilde M_2(\phi) = M_0(\phi)+\epsilon M_1(\phi) + \mathcal O(\epsilon^2)$ determines the main terms of a diagonaliser of the matrix $A(\epsilon,\phi)$ and 
we can deduce the statements about the eigenvalue asymptotics. 
\end{proof}


\subsubsection{System formulation} Let $M(\epsilon,\phi)$ be the diagonaliser 
of $A(\epsilon,\phi)$ constructed in Proposition~\ref{prop:3.1}. Then we consider  
\begin{equation}\label{eq:3.10}
  V(t,\epsilon,\phi,|\xi|) = \begin{pmatrix}
  (\D_t+|\xi|\mathcal D^{1/2}(\epsilon,\phi))M^{-1}(\epsilon,\phi) \hat U(t,\xi)\\
  (\D_t-|\xi|\mathcal D^{1/2}(\epsilon,\phi))M^{-1}(\epsilon,\phi) \hat U(t,\xi)\\
 \hat \theta 
  \end{pmatrix} \in \C^{7},
\end{equation}
with $\xi=|\xi| \eta(\epsilon,\phi)$ and 
$\mathcal D^{1/2}(\epsilon,\phi)=\diag(\omega_1(\epsilon,\phi),...)$ the 
diagonal matrix containing the square roots 
$\omega_j(\epsilon,\phi)=\sqrt{\varkappa_j(\epsilon,\phi)}$ of the 
eigenvalues of $A(\epsilon,\phi)$. The vector $V$ satisfies the first order system
$\D_tV=B(\epsilon,\phi,|\xi|) V$ with $B(\epsilon,\phi,|\xi|) =B_1(\epsilon,\phi)|\xi| +B_2|\xi|^2 $ given by
\begin{equation}\label{eq:3.11}
B_1(\epsilon,\phi) = \begin{pmatrix} 
\omega_1(\epsilon,\phi) &&&&\i\gamma a_1(\epsilon,\phi)\\
&\omega_2(\epsilon,\phi) &&&\i\gamma a_2(\epsilon,\phi)\\
&&\ddots & &\vdots \\
&&&-\omega_3(\epsilon,\phi) & \i\gamma a_3(\epsilon,\phi) \\
-\frac\i 2\gamma a_1(\epsilon,\phi) & -\frac\i2\gamma a_2(\epsilon,\phi) & 
\cdots& -\frac\i2 \gamma a_3(\epsilon,\phi) & 0\end{pmatrix} 
\end{equation} 
and $B_2=\diag(0,\ldots,0,\i\kappa)$. The coupling functions 
$a_j(\epsilon,\phi)$ are the components of the vector $M^{-1}(\epsilon,\phi)\eta$. 
From Proposition~\ref{prop:3.1} we know that they have asymptotic expansions as $\epsilon\to0$.

\begin{rem}
{\sl 1.} Since $M^{-1}(\epsilon,\phi)=\tilde M_2^*(\phi)(I-\epsilon N_1^{(1)}(\phi))\tilde M^*+\mathcal O(\epsilon^2)$ by our construction it follows that 
\begin{subequations}
\begin{align}
a_1(\epsilon,\phi)&=1+\mathcal O(\epsilon^2),\\
a_2(\epsilon,\phi)&=\epsilon\frac{2(\mu+2\lambda-\tau)}{3(\lambda+\mu)}\sin\frac{3\phi}2 +\mathcal O(\epsilon^2),\\
a_3(\epsilon,\phi)&=\epsilon\frac{2(\mu+2\lambda-\tau)}{3(\lambda+\mu)}\cos\frac{3\phi}2 +\mathcal O(\epsilon^2).
\end{align}
\end{subequations}
We know that the coupling functions vanish along three great circles 
through $\bar\eta$. We see that the numbering of the eigenprojections 
is not consistent along the circles. The coupling functions $a_2$ and 
$a_3$ vanish both in the degenerate direction.\\
{\sl 2.} Since we do not assume that $M(\epsilon,\phi)$ is unitary the 
relation $\sum_j a_j^2=1$ does not hold  for these coupling functions. 
However, $M_0(\phi)$ is unitary and therefore 
$\sum_j a_j^2=1+\mathcal O(\epsilon)$ as already observed.
\end{rem}

\subsubsection{Real and imaginary parts of eigenvalues} The coefficient matrix 
$B(\epsilon,\phi,|\xi|)$ has the same structure as $B(\xi)$ in 
Section~\ref{sec2}. Therefore, we can conclude similar statements on 
eigenvalues and their behaviour by (a) investigating the characteristic 
polynomial and (b) block-diagonalising for small and large $|\xi|$, respectively. 

\begin{prop}
\begin{enumerate}
\item  $\trace B(\epsilon,\phi,|\xi|) = \i\kappa|\xi|^2$ and 
$\det B(\epsilon,\phi,|\xi|) = \i\kappa|\xi|^2\det A(\xi)$.
\item $B(\epsilon,\phi,|\xi|)$ has purely real eigenvalues for $|\xi|\ne0$ if and only if 
$a_2(\epsilon,\phi)a_3(\epsilon,\phi)=0$, i.e., $\epsilon=0$ or $\phi\in\frac\pi3\mathbb Z$. 
\item  $B(0,\phi,|\xi|)$ has the real eigenvalues 
$\pm\omega_{2,3}(0,\phi)=\frac{\sqrt3}3(\tau+\mu-\lambda)$ and three 
eigenvalues satisfying $\Im\nu \ge C$ if $|\xi|\ge c$ and $\Im\nu\sim|\xi|^2$ if $|\xi|<c$.
\item The quotient
\begin{equation}\label{eq:deg-dir-quotient}
\frac{a_2^2(\epsilon,\phi)(\nu_{2,3}^2(\epsilon,\phi,|\xi|)-\varkappa_3(\epsilon,\phi)|\xi|^2) + a_3^2(\epsilon,\phi)(\nu_{2,3}^2(\epsilon,\phi,|\xi|)-\varkappa_2(\epsilon,\phi)|\xi|^2)}{(\nu_{2,3}^2(\epsilon,\phi,|\xi|)-\varkappa_{2}(\epsilon,\phi)|\xi|^2)(\nu_{2,3}^2(\epsilon,\phi)-\varkappa_3(\epsilon,\phi)|\xi|^2)}
\end{equation}
involving the hyperbolic eigenvalues $\nu_{2,3}^\pm$ of $B(\epsilon,\phi,|\xi|)$ 
is smooth and non-vanishing for fixed values of $|\xi|$.
\end{enumerate}
\end{prop}
\begin{proof}
We consider only part (2) to (4). The characteristic polynomial of $B$ is 
given by an expression like \eqref{eq:CharPolB}. If we assume that eigenvalues 
are purely real we can split the expression into real and imaginary part. We 
consider the imaginary part first, which leads to
\begin{equation}
  \kappa|\xi|^2\prod_{j=1}^3 (\nu^2 -\varkappa_j(\epsilon,\phi)|\xi|^2) = 0.
\end{equation}
Therefore, real eigenvalues have to coincide with the square roots of 
eigenvalues of $A(\xi)$. If we assume $\nu^2=\varkappa_j(\epsilon,\phi)|\xi|^2)$ 
is a root of the characteristic equation, we can divide by the corresponding 
factor and obtain if $\epsilon\ne 0$ (and therefore $A$ is non-degenerate)
\begin{equation}
  a_j^2(\epsilon,\phi) = 0.
\end{equation}
If $\epsilon=0$ the characteristic polynomial factors as
\begin{equation}
  ((\nu-\i\kappa|\xi|^2) (\nu^2-\bar\varkappa_1|\xi|^2)-\nu\gamma^2|\xi|^2 ) (\nu-\bar\varkappa_{2,3}|\xi|^2)^2=0
\end{equation}
with $\bar\varkappa_1=\frac13(\tau+2\lambda+4\mu)$ and 
$\bar\varkappa_{2,3} = \frac13 (\tau+\mu-\lambda)$. The first factor 
resembles one-dimensional thermo-elasticity (with $\tau^2=\bar\varkappa_1$) 
and gives three roots with positive imaginary parts subject to 
\eqref{eq:2.43} and \eqref{eq:2.44}. Finally (4) follows by collecting
the two related terms in the characteristic equation of form \eqref{eq:KelvinCharEq}.
The imaginary part of the quotient is given by $-\kappa|\xi|^2/\nu_{2,3}^\pm$ 
in hyperbolic/degenerate directions and therefore non-zero.
\end{proof}

The quotient \eqref{eq:deg-dir-quotient} may be used to determine asymptotic 
expansions of the hyperbolic eigenvalue and its imaginary part as $\epsilon\to0$ 
for fixed $|\xi|$ and $\phi\not\in\frac\pi3\mathbb Z$. We will 
follow a different strategy and diagonalise as $\epsilon\to0$ uniform on bounded $\xi$.

\subsubsection{Asymptotic expansion as $\epsilon\to 0$ uniform in $|\xi|$}
Note first, that $B(|\xi|,0,\phi)$ is independent of $\phi$ and just 
the system of one-dimensional thermo-elasticity \eqref{eq:B-1d} extended 
by four additional diagonal entries. Since we need to understand this 
system first, we are going to recall some facts about the one-dimensional 
theory. As $|\xi|$ becomes small/large we already gave asymptotic 
expansions of eigenvalues in Section~\ref{ssec:1d}.  The bit of information 
which is still missing is contained in the following lemma.

\begin{lem}
The coefficient matrix $B(\xi)$ of the one-dimensional thermo-elastic 
system given in \eqref{eq:B-1d} has for $\xi\ne0$ and under the 
natural assumptions $\gamma,\kappa,\tau>0$ only simple eigenvalues. 
\end{lem}
\begin{proof}
Note that the characteristic polynomial of this matrix $B(\xi)$ is given by 
$$
\nu^3-\i\kappa|\xi|^2\nu^2+\tau^2|\xi|^2\nu +\i\tau^2\kappa|\xi|^4,
$$
which is invariant under the transform $\nu\mapsto \overline{-\nu}$ 
and has alternating imaginary and real coefficients. From that we 
conclude that the only possible solutions are of the form $\i a$, $b+\i c$ 
and $-b+\i c$ for certain real $a,b,c$. Furthermore, from the general 
theory of Section~\ref{sec2} it is clear that $a,c>0$. Thus, the only 
possibility for multiplicities to occur is if $b=0$. Plugging in $b=0$ 
and multiplying the corresponding linear factors gives
$$ \nu^3-\nu^2 (\i a +2\i c) -\nu(c^2+2ac) +\i c^2 a. $$
Comparing coefficients with the above polynomial implies that 
$\kappa|\xi|^2 = - ca / (c+2a)$, which contradicts to the positivity 
of all quantities involved.
\end{proof}

We write the coefficient matrix $B(|\xi|,\epsilon,\phi)$ as sum 
of homogeneous components in $\epsilon$
\begin{align}\label{eq:3.17}
  &|\xi|^{-1} B(|\xi|,\epsilon,\phi)  = B^{(0)}(|\xi|,\phi) + \epsilon B^{(1)}(|\xi|,\phi) + \mathcal O(\epsilon^2),\\ 
\intertext{where}
  B^{(0)}(|\xi|,\phi) &=  \begin{pmatrix}
    \bar\omega_1&&&&&& \i\gamma \\
     &\bar\omega_2 &&&&&\\
     &&\bar\omega_2&&&\\
     &&&- \bar\omega_1&&& \i\gamma\\
     &&&&-\bar\omega_2\\
     &&&&&-\bar\omega_2\\
     -\frac\i2\gamma&&&-\frac\i2\gamma&&&\i\kappa|\xi|
   \end{pmatrix},\\
  \small B^{(1)}(|\xi|,\phi) &= \small \begin{pmatrix}
   0&&&&&&0\\&\delta_1&&&&& \i\gamma\delta_2\sin\frac{3\phi}2\\
   &&-\delta_1&&&&\i\gamma\delta_2\cos\frac{3\phi}2\\
   &&&0&&&0\\&&&&\delta_1&&\i\gamma\delta_2\sin\frac{3\phi}2\\
   &&&&&-\delta_1&\i\gamma\delta_2\cos\frac{3\phi}2\\
   0&-\frac{\i\gamma\delta_2}2\sin\frac{3\phi}2&-\frac{\i\gamma\delta_2}\cos\frac{3\phi}2&0&-\frac{\i\gamma\delta_2}2\sin\frac{3\phi}2&-\frac{\i\gamma\delta_2}2\cos\frac{3\phi}2&0
   \end{pmatrix},
\end{align}
and $\bar\omega_1= \sqrt{\frac{\tau+2\lambda+4\mu}{3}} $, $\bar\omega_2=\sqrt{\frac{\tau+\mu-\lambda}3}$,
$\delta_1=\frac1{\sqrt6} \frac{-\tau+2\mu+\lambda}{\sqrt{\tau+\mu-\lambda}}$ and $\delta_2=\frac{2(\mu+2\lambda-\tau)}{3(\lambda+\mu)}$. 

As a direct consequence of the previous lemma in combination with the 
asymptotic expansions of Section~\ref{ssec:1d} we obtain

\begin{prop}\label{prop:3.4}
Assume, that $\lambda+\mu\ne0$ and $\gamma^2+\lambda+\mu\ne0$. Then 
the matrix $B^{(0)}(|\xi|,\phi)$ has uniformly separated eigenvalues 
in $|\xi|\in\R$, $\phi\in\S^1$ (where $\pm\bar\omega_2$
are of constant multiplicity two). 
\end{prop}

Now we can apply several steps of diagonalisation
based on the scheme of \cite[Sec. 2]{JaWi}. At first we apply 
the diagonaliser of the main part. This has only effects on the two 
entries related to $\bar\omega_1$  and the last row/column and 
determines the eigenvalues $\nu_0(|\xi|,\epsilon,\phi)$ and 
$\nu_1^\pm(|\xi|,\epsilon,\phi)$  modulo $\epsilon^2$.  Furthermore, 
Proposition~\ref{prop:3.4} allows to $(1,2,1,2,1)$-block-diagonalise 
modulo $\mathcal O(\epsilon^N)$, $N$ arbitrary. 

Finally we can investigate the remaining $2\times2$-blocks and 
diagonalise again because the $\epsilon$-homogeneous entries 
$\pm\delta_1\epsilon$ are distinct (trivially uniform in $|\xi|$ and $\phi$).

\begin{prop}\label{prop:3.5} Assume, that $\lambda+\mu\ne0$ and 
$\gamma^2+\lambda+\mu\ne0$. The eigenvalues of $B(|\xi|,\epsilon,\phi)$ 
have uniformly   in $|\xi|$ and $\phi$ full asymptotic expansions as 
$\epsilon\to0$. The first main terms are given as
\begin{subequations}
\begin{align}
  \nu_0(|\xi|,\epsilon,\phi)&=  \check \nu_0(|\xi|) +  |\xi|\mathcal O(\epsilon^2),\label{eq:3.20a}\\
  \nu_1^\pm(|\xi|,\epsilon,\phi)&=  \check \nu_1^\pm (|\xi|) + |\xi|\mathcal O(\epsilon^2),\\ 
  \nu_{2/3}^{\pm_1,\pm_2}(|\xi|,\epsilon,\phi) &=\pm_1 \bar\omega_2|\xi| \pm_2 \delta_1|\xi|\epsilon +|\xi|\mathcal O(\epsilon^2)
\end{align}
where $\check \nu_0(|\xi|)$ and $\check \nu_1^\pm(|\xi|)$ are 
the eigenvalues of the one-dimensional thermo-elastic system 
with propagation speed $\bar\omega_1$ and the signs $\pm_1$ and 
$\pm_2$ are independent of each other.
\end{subequations}
\end{prop}

\begin{rem}
The statement holds true uniform in $|\xi|$. However, it is only of 
use as long as the error terms $|\xi|\epsilon^N$ are smaller than the 
size of the eigenvalues. For $|\xi|\to0$ the eigenvalues of the 
one-dimensional thermo-elastic system behave like 
$\check\nu_0(|\xi|)\sim|\xi|^2$ and $\check\nu_1^\pm(|\xi|)\sim \pm |\xi|$. 
Hence, the statement of \eqref{eq:3.20a} is only of use 
if $|\xi|\epsilon^2 \ll |\xi|^2$, i.e. if  $\epsilon^2\ll |\xi|$.
For $|\xi|\to\infty$ we know similarly $\check\nu_0(|\xi|)\sim |\xi|^2$ 
and $\check\nu_1^\pm(|\xi|)\sim\pm|\xi|$, which in turn implies
that the expansion determines the behaviour of the eigenvalues.

This restriction is by no means a severe one; the expansion is 
only of interest for the 'degenerate' eigenvalues 
$\nu_{2/3}^{\pm_1,\pm_2}(|\xi|,\epsilon,\phi)$ (for which no such 
restriction appears).
\end{rem}

\subsubsection{Diagonalisation for small and large $|\xi|$.} To complete the picture we 
want to give some comments on expansions for small and large values of $|\xi|$ 
under the same assumptions as in Proposition~\ref{prop:3.5}. Using the ideas 
from \cite{Wirth:2008b} we can employ the (block) diagonalisation scheme to 
separate the three non-degenerate eigenvalues from the two degenerate ones 
asymptotically and give full asymptotic expansions for them as $|\xi|$ tends 
to zero or infinity. The obtained expressions coincide with the formulae from 
Propositions~\ref{prop:2.5} and~\ref{prop:2.6}. It remains to understand the 
behaviour of the remaining $2\times 2$-blocks. This can be done directly by 
solving the characteristic polynomial as in \cite[Prop.~2.7]{RW07} or by a 
second diagonalisation scheme. 

We focus on the latter idea and consider the case of small $|\xi|$ first. 
The $2\times 2$-blocks have the form
\begin{equation}
  f_0(|\xi|,\epsilon,\phi) I + \begin{pmatrix} \delta_0(|\xi|,\epsilon,\phi)\\&-\delta_0(|\xi|,\epsilon,\phi) \end{pmatrix} + \mathcal O(|\xi|^2)
\end{equation}
with a function $\delta_0(|\xi|,\epsilon,\phi)\sim \epsilon|\xi|$. If we restrict the consideration to the zone 
\begin{equation}
  \mathcal Z_0 (c) =\{ (|\xi|,\epsilon,\phi)\,:\, |\xi|\le c\epsilon,\,\epsilon\ll1\},
\end{equation}
the remainder can be written as $\epsilon|\xi| \mathcal O(\epsilon^{-1}|\xi|)$ and the standard diagonalisation scheme applied to the last two terms gives full asymptotic expansions in powers of $\epsilon^{-1}|\xi|$ as $\epsilon^{-1}|\xi|\to0$, 
\begin{equation}
  f_0(|\xi|,\epsilon,\phi) \pm \delta_0(|\xi|,\epsilon,\phi) +   .... + \epsilon|\xi| \mathcal O(\epsilon^{-N}|\xi|^N).
\end{equation}
A similar idea applies for large $|\xi|$ in the zone
\begin{equation}
  \mathcal Z_\infty (N) =\{ (|\xi|,\epsilon,\phi)\,:\, \epsilon |\xi|\ge N,\,\epsilon\ll1\}.
\end{equation}
Based on 
\begin{equation}
  f_\infty(|\xi|,\epsilon,\phi) I + \begin{pmatrix} \delta_\infty(|\xi|,\epsilon,\phi)\\&-\delta_\infty(|\xi|,\epsilon,\phi) \end{pmatrix} + \mathcal O(1)
\end{equation}
with a function $\delta_\infty(|\xi|,\epsilon,\phi)\sim \epsilon|\xi|$ it gives asymptotic expansions in powers of $\epsilon|\xi|$ as $\epsilon|\xi|\to\infty$. 
 
\subsection{Cubic media, uniplanar singularities} The Fresnel surface for cubic media has
six uniplanar singularities. Again they are equivalent and it suffices to consider the 
neighbourhood of $\bar\eta=(1,0,0)^T\in\S^2$.

We introduce polar co-ordinates near $\bar\eta$. Let $\epsilon\geq0$ and 
$\phi\in[-\pi,\pi)$. Then we set 
\begin{equation}
\eta = \begin{pmatrix} \eta_1\\\eta_2\\\eta_3 \end{pmatrix}
=\sqrt{1-\epsilon^2} \begin{pmatrix} 1\\0\\0\end{pmatrix}
 +\epsilon\cos\phi\begin{pmatrix}0\\1\\0\end{pmatrix}+ \epsilon\sin\phi
 \begin{pmatrix}0\\0\\1\end{pmatrix},
\end{equation}
and use an asymptotic expansion of $A(\eta)$ as $\epsilon\to0$ 
\begin{equation}
   A(\eta) = A_0+ \epsilon A_1(\phi) + \epsilon^2 A_2(\phi) + \mathcal O(\epsilon^3)
 \end{equation}
 with coefficients
 \begin{subequations}
 \begin{align}
 A_0&= \diag(\tau,\mu,\mu)\\
  A_1(\phi)&= (\lambda+\mu) \begin{pmatrix}&\cos\phi&\sin\phi\\\cos\phi\\\sin\phi\end{pmatrix} \\
  A_2(\phi) &= (\tau-\mu) \begin{pmatrix}-1\\&\cos^2\phi\\&&\sin^2\phi\end{pmatrix}
  + \frac{\lambda+\mu}2   \begin{pmatrix} 0 \\ &  &  \sin2\phi\\& \sin2\phi \end{pmatrix}
 \end{align}
\end{subequations}
to deduce properties of the eigenvalues and eigenprojections of $A(\eta)$ near 
$\bar\eta$.  We restrict considerations to the case when $\tau\ne\mu$. Then the 
following statement follows again by the two-step diagonalisation procedure 
(like in the conical case and as developed in \cite{JaWi}, \cite{RW07}).
\begin{prop} \label{prop:3:6}
Assume $\lambda+\mu\ne0$, $\tau\ne\mu$ and $\tau \ne \lambda+2\mu$.
Then the eigenvalues $\varkappa_j(\eta)$ and the corresponding eigenprojections 
have uniformly in $\phi$ full asymptotic expansions as $\epsilon\to0$. 
The main terms are given by
\begin{subequations}\label{eq:3:29}
\begin{align}
\varkappa_1(\eta) &= \tau - C \epsilon^2 + \mathcal O(\epsilon^3),\\
\varkappa_2(\eta) &= \mu + \frac12\left(C+\sqrt{C^2-(C^2-D^2)\sin^2(2\phi)} \right)\epsilon^2 + \mathcal O(\epsilon^3),\\
\varkappa_3(\eta) &= \mu + \frac12\left(C-\sqrt{C^2-(C^2-D^2)\sin^2(2\phi)} \right)\epsilon^2 + \mathcal O(\epsilon^3),
\end{align}
\end{subequations}
where
\begin{equation}\label{eq:3:30}
C=\frac{(\tau-\mu)^2-(\lambda+\mu)^2}{\tau-\mu},
\qquad D=\lambda+\mu 
\end{equation}
\end{prop}
\begin{rem}
This statement reflects what we mean by an uniplanar singularity. Two of 
the eigenvalues coincide up to second order.
\end{rem}

\begin{proof}
We follow the diagonalisation scheme. $A_0$ is already diagonal,  $A_1$ 
does not contain (1,2)-block diagonal entries. To get expansions for the 
eigenvalues we have to apply two steps of block-diagonalisation. First we 
treat $A_1$ by the aid of 
\begin{equation}
   N_1^{(1)}(\phi) = \frac{\lambda+\mu}{\tau-\mu} \begin{pmatrix} & -\cos\phi & -\sin\phi\\\cos\phi\\\sin\phi \end{pmatrix},
\end{equation}
such that $I+\epsilon N_1^{(1)}(\phi)$ block-diagonalises the matrix 
modulo $\epsilon^2$. This preserves $A_0$ and $0=\bdiag_{1,2} A_1$ and 
gives the new $2$-homogeneous component
\begin{equation}
  A_2+A_1 N_1^{(1)},\qquad A_1N_1^{(1)} = \frac{(\lambda+\mu)^2}{\tau-\mu} \diag (1,-\cos^2\phi, -\sin^2\phi).
\end{equation}
The starting terms of the expansion of the first eigenvalue can be read 
directly from these matrices. For the remaining two we have to diagonalise 
the lower $2\times 2$ block. This block has the form
\begin{equation}\label{eq:3:33}
\begin{pmatrix}
   C \cos^2\phi & D \sin\phi\cos\phi \\ D \sin\phi\cos\phi & C \sin^2\phi 
\end{pmatrix}
\end{equation}
with $C$, $D$ from \eqref{eq:3:30}. Eigenvalues of this matrix are 
uniformly separated if the condition
\begin{equation}
C^2 > (C^2-D^2)\sin^2(2\phi),\qquad\text{i.e.}\quad C\ne0,\quad D \ne 0
\end{equation}
is satisfied. Under this assumption the full diagonalisation scheme 
works through and the main terms can be calculated directly and give 
\eqref{eq:3:29}. For completeness we also give a unitary diagonaliser 
of the matrix \eqref{eq:3:33}, namely
\begin{align}
M_2(\phi) &=\frac1{\sqrt{2D^2\sin^2(2\phi)+2C^2\cos^2(2\phi)+2C\cos2\phi\sqrt{\;:}}}\notag
	\\ &\qquad\qquad \times \begin{pmatrix}
 C\cos2\phi+\sqrt{\;:} &- D\sin2\phi\\D\sin2\phi & C\cos2\phi+\sqrt{\;:}
\end{pmatrix}\notag\\
&=\begin{pmatrix}m_{1}(\phi) & m_{2}(\phi) \\ -m_{2}(\phi)& m_{1}(\phi)\end{pmatrix}
\end{align} 
where $\sqrt{\;:}=\sqrt{C^2-(C^2-D^2)\sin^2(2\phi)}$, $\phi\ne \frac\pi2,\frac{3\pi}2$. 
Expressions are extended by continuity.
\end{proof}

\subsubsection{System form}
Again we use the diagonaliser $M(\epsilon,\phi)$ of $A(\epsilon,\phi)$ 
constructed in Proposition~\ref{prop:3:6} to reformulate the thermo-elastic 
system as system of first order. Formulae \eqref{eq:3.10} and 
\eqref{eq:3.11} give the corresponding representation.

\begin{rem}
{\sl 1.}
Since $M^{-1}(\epsilon,\phi)=\diag(1,M_2^{*}(\phi)) (I-\epsilon N_1^{(1)}(\phi)) +\mathcal O(\epsilon^2)$ in the notation of the proof of Proposition~\ref{prop:3:6} 
it follows that the coupling functions satisfy
\begin{subequations}
\begin{align}
  a_1(\epsilon,\phi) &= 1 +\mathcal O(\epsilon^2)\\
  a_2(\epsilon,\phi) &= \epsilon \frac{\tau-\lambda-2\mu}{\tau-\mu} (m_1(\phi)\cos\phi+m_2(\phi)\sin\phi)  +\mathcal O(\epsilon^2)\\
  a_3(\epsilon,\phi) &= \epsilon \frac{\tau-\lambda-2\mu}{\tau-\mu} (m_1(\phi)\sin\phi-m_2(\phi)\cos\phi)    +\mathcal O(\epsilon^2)
\end{align}  
\end{subequations}
Since $\tau\ne\lambda+2\mu$ the function $a_2(\phi)$ vanishes only  
for $\phi=k\frac\pi2$, $k\in\mathbb Z$, while $a_3(\phi)$ vanishes 
for $\phi=(2k+1)\frac\pi4$, $k\in\mathbb Z$. 
\\ 
{\sl 2.} Note that in contrast to the conic situation the eigenvalues 
coincide to second order in the degenerate direction, while the coupling 
functions still vanish to first order (if we approach the degeneracy from 
parabolic directions).
\end{rem}

\subsubsection{Asymptotic expansion of eigenvalues as $\epsilon\to0$} We write 
the coefficient matrix
$B(|\xi|,\epsilon,\phi)$ as sum of homogeneous components in $\epsilon$, 
cf. \eqref{eq:3.17}. This gives
\begin{align}
  B^{(0)}(|\xi|,\phi)&= \begin{pmatrix} \sqrt\tau && & & & & \i\gamma \\ 
  & \sqrt\mu \\
  && \sqrt\mu \\
  &&& -\sqrt\tau &&& \i\gamma \\
  &&&&-\sqrt\mu\\
  &&&&&-\sqrt\mu \\
  -\frac\i2\gamma &&& -\frac\i2 \gamma &&& \i\kappa|\xi| 
  \end{pmatrix}\\
  B^{(1)}(|\xi|,\phi) &= \begin{pmatrix} 
  && & & & & 0\\
  && & & & &\i\gamma a_2^{(1)}(\phi)\\
  && & & & &\i\gamma a_3^{(1)}(\phi)\\
  && & & & & 0\\
  && & & & &\i\gamma a_2^{(1)}(\phi)\\
  && & & & &\i\gamma a_3^{(1)}(\phi)\\
  0 & \frac\i2 \gamma a_2^{(1)}(\phi)& \frac\i2 \gamma a_3^{(1)}(\phi)&  0 & \frac\i2 \gamma a_2^{(1)}(\phi)& \frac\i2 \gamma a_3^{(1)}(\phi)& 0\\  
  \end{pmatrix}
\end{align}
and $B^{(2)}(|\xi|,\phi)$ has entries on the diagonal, in the last 
row and last column. In order to apply a block-diagonalisation as 
$\epsilon\to0$ we assume that the matrix $B^{(0)}(|\xi|,\phi)$ has 
as many distinct eigenvalues as possible. This is ensured if 
$\mu\ne\tau$, $\mu\ne\tau+\gamma^2$ and we can
(1,2,1,2,1)-block-diagonalise this matrix family. Note, that due to the structure of the last
rows and columns, the system decouples modulo $\epsilon^2$ into a one-dimensional thermo-elastic system and one containing the elastic eigenvalues. The coupling comes only into
play for the $\epsilon^3$ entries.

\begin{prop}\label{prop:3.7}
Assume $\mu\ne\tau$, $\mu\ne\tau+\gamma^2$. Then the eigenvalues and 
eigenprojections of $B(|\xi|,\epsilon,\phi)$ have full asymptotic 
expansions as $\epsilon\to0$. The main terms are given by
\begin{subequations}
\begin{align}
\nu_0(|\xi|,\epsilon,\phi) &= \check\nu_0(|\xi|) + |\xi|\mathcal O(\epsilon^3),\\
\nu_1^\pm(|\xi|,\epsilon,\phi) &= \check\nu_1^\pm(|\xi|) + |\xi|\mathcal O(\epsilon^3),\\
\nu_{2/3}^{\pm_1,\pm_2} (|\xi|,\epsilon,\phi) &= \pm_1 \sqrt{\mu} |\xi| +  \frac{C \pm_2 \sqrt{C^2-(C^2-D^2)\sin^2(2\phi)} }{4\sqrt\mu}|\xi|\epsilon^2 +   |\xi| \mathcal O(\epsilon^3)
\end{align}
\end{subequations}
where $\check\nu_0(|\xi|)$ and $\check\nu_1^\pm(|\xi|)$ are eigenvalues of the one-dimensional thermo-elastic system with parameter $\sqrt\tau$. The signs $\pm_1$ and $\pm_2$ are independent and the parameters $C$ and $D$ are given by \eqref{eq:3:30}.
\end{prop}

\begin{rem} Similar to Proposition~\ref{prop:3.5} this statement is uniform in $|\xi|$. It will
be of particular importance for us that the hyperbolic eigenvalues $\nu_{2/3}^\pm$ coincide
up to second order in $\epsilon$ with the corresponding (roots of) eigenvalues of the elastic
operator. This will be the key observation to transfer stationary phase estimates
from elastic systems to the thermo-elastic one.
\end{rem}

\subsection{Hexagonal media}\label{sec3.3} Finally we want to discuss the 
case of hexagonal media. The elastic operator defined by 
\eqref{eq:1.4}--\eqref{eq:1.5} is invariant under rotations in $x_3$-direction.
We will make use of this fact and reduce considerations to a two-dimensional 
situation corresponding to cross-sections of the Fresnel surface. 

\begin{figure}
\includegraphics[width=5cm]{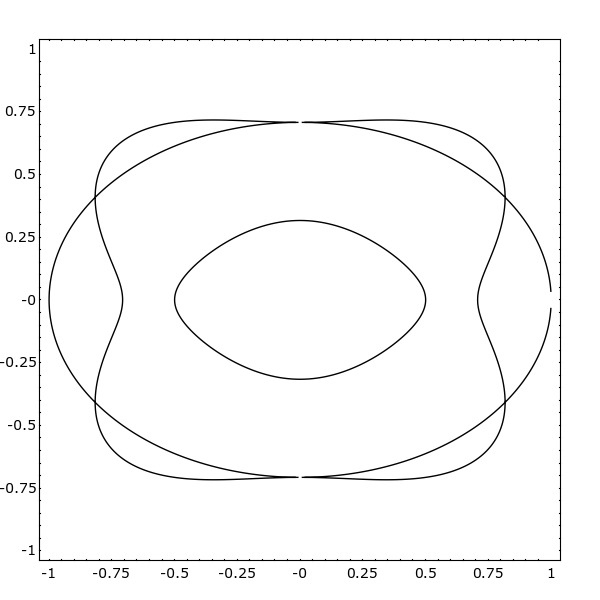}
\caption{Cut of the Fresnel surface for hexagonal media, $\eta_2=0$. The parameters are chosen as 
$\tau_1=4$, $\tau_2=10$, $\lambda_1=2$, $\lambda_2=4$, $\mu=2$. The complete surface is obtained by rotation along the vertical axis.}
\end{figure}

As already pointed out in Section~\ref{sec:2:hex} degenerate directions are
$\pm(0,0,1)^\top$, which are uniplanar. They could be handled similar to the 
cubic case, but rotational invariance makes estimates simpler. 
There are further circles of degenerate directions if 
\begin{equation}\label{eq:3.40}
   {\tau_2      - 2\tau_1}\ge  {\lambda_2 +       2 \mu}.
\end{equation}
We exploit rotational symmetry and consider the system only in the frequency
hyperplane $\eta_1=0$. Then it is possible to express the eigenvectors 
$r_j(\eta)$ corresponding to eigenvalues $\varkappa_j(\eta)$ smoothly
and the thermo-elastic system can be reformulated as system of first order
in full analogy to the general treatment in Section~\ref{sec2}. The derived
asymptotic expansions for eigenvalues and the description of their behaviour
transfers away from the degeneracy and has to be equipped with an additional 
description near these degenerate directions. 


Apart from the $\xi_3$-axis it is possible to find smooth families of
eigenvectors $r_j(\eta)$ of $A(\eta)$. This follows directly from rotational
invariance combined with one-dimensional perturbation theory of matrices,
\cite{Kato:1980}. If we assume that the frequency support of initial data
and therefore of the solution $U$, $\theta$ is conically separated from
the uniplanar directions we can follow Section~\ref{sec2} and rewrite as 
first order system in $V(t,\xi)$ with coefficient matrix $B(\xi)$ given 
by \eqref{eq:2.9} and of (1,1,5)-block structure. In what follows, we will
ignore the scalar hyperbolic blocks and consider the remaining $5\times 5$ 
matrix.

Based on the discussions from Section~\ref{sec:2:hex} we know that this 
$5\times5$ block is non-degenerate in the sense that its $1$-homogeneous 
part has distinct eigenvalues if $(\tau_1-\mu)(\tau_2-\mu)\ne0$. We assume
this in the sequel. But this means that the theory of Section~\ref{sec2}
is applicable and gives a full discription of eigenvalues and 
eigenprojections and we are done. 

Near the uniplanar directions, i.e., on the $\xi_3$-axis, we follow
the same approach as for cubic media. We introduce polar co-ordinates
around this direction and construct expressions for the corresponding
asymptotics. There is one major simplification,
due to rotational invariance the construction is independent of the angular
variable.

\section{Dispersive estimates}\label{sec4}
We will show how to use the information obtained in Sections~\ref{sec2} 
and~\ref{sec3} to derive $L^p$--$L^q$ decay estimates for solutions to 
thermo-elastic systems. Some of the ideas we present are general in the 
sense that they can be applied to arbitrary space dimensions, however, 
our moin focus will be the three-dimensional case and the examples 
considered in Section~\ref{sec3}. 

The estimates we have in mind are micro-localised to (a) non-degenerate  
parabolic, (b) non-degenerate hyperbolic or (c) degenerate directions. 
The first two situations generalise the consideration of \cite{RW07}, 
\cite{Wirth:2007c} taking also into account the estimates due to Sugimoto \cite{Sug94}, \cite{Sug96}, while the treatment of degenerate directions is 
inspired by the work of Liess \cite{Lie91}, \cite{Lie04}, \cite{Lie06}.

\subsection{Non-degenerate directions}\label{sec41}  We will consider two 
situations and micro-localise solutions to either open sets of parabolic 
directions or tubular neighbourhoods of compact parts of regular submanifolds 
of hyperbolic directions. 

\subsubsection{Estimates in parabolic directions and for parabolic modes}
Let first $\psi,\tilde \psi\in C_0^\infty (\S^{n-1})$ be supported in 
$\mathcal U$ with $\tilde\psi=1$ on $\supp\psi$ and $\chi\in C^\infty(\R_+)$ 
a cut-off function satisfying $\chi(s)=0$ for $s\le\epsilon$ and $\chi(s)=1$ 
for $s\ge 2\epsilon$. We extend both $\psi$ and $\tilde\psi$ as 
0-homogeneous functions to $\R^n$. Then we consider the solution to the 
first order system
\begin{equation}\label{eq:4.1}
   \D_t V  = B(D) V,\qquad V(0,\cdot)=\tilde \psi(D)  V_0,
\end{equation}
with data $V_0\in\mathcal S(\R^n;\C^{2n+1})$. Note, that this is 
well-defined and $B(\xi)$ needs only to be defined on $\supp\tilde\psi$.  

\begin{lem}[Parabolic estimate]\label{lem:4.1}
Assume that $\supp\psi$ is contained in the set of parabolic directions.
Then the solutions to \eqref{eq:4.1} satisfy the a-priori estimates
\begin{subequations}
\begin{align}
  \|\chi(|D|)\psi(D) V(t,\cdot)\|_q &\lesssim \e^{-Ct} \| V_0\|_{p,r},\\
  \|(1-\chi(|D|))\psi(D) V(t,\cdot)\|_q &\lesssim (1+t)^{-\frac n2(\frac1p-\frac1q)} \|V_0\|_p
\end{align}  
\end{subequations}
for all $1\le p\le2\le q\le\infty$ and with Sobolev regularity $r>n(1/p-1/q)$.
\end{lem}
\begin{proof}[Sketch of proof]
The proof of this estimate is straightforward from the two-dimensional 
situation considered in \cite{RW07}. For small frequencies we write the solution $V$ as sum 
\begin{equation}\label{eq:4.3}
   V(t,x) = \sum_{\nu(\xi)\in\spec B(\xi)} \e^{\i t \nu(\D)} P_{\nu}(\D) V_0,
\end{equation}
$P_\nu$ the corresponding eigenprojections. We know that 
$\|P_\nu(\xi)\|\lesssim 1$ and $\Im \nu(\xi)\sim |\xi|^2$ by Proposition~\ref{prop:2.5}. 
Now each of the appearing terms can be estimated using the $L^p$--$L^{p'}$ 
boundedness of the Fourier transform (for $p\in[1,2]$) and H\"older inequality. 
Similarly, the representation \eqref{eq:4.3} in combination with the bound $\Im \spec B(\xi) \ge C$ gives  
exponential decay of $L^2$ and $H^s$ norms and this combined with Sobolev 
embedding yields the desired estimate. 

For intermediate frequencies we may have to deal with multiplicities and
resulting singularities of the spectral projections. Instead of \eqref{eq:4.3} we use a spectral calculus representation which implies  
\begin{equation}
   | \hat V(t,\xi) | \le \e^{-C t}  \frac1{2\pi} \int_\Gamma \|(\zeta-B(\xi))^{-1}\|\d \zeta \lesssim \e^{-Ct}
\end{equation}
based on the compactness of the relevant set of frequencies $\xi$ and the 
bound on the imaginary part due to Corollary~\ref{cor:2.8} / Proposition~\ref{prop:2.9}. 
Here, $\Gamma$ is a smooth curve encircling the family of parabolic eigenvalues for the relevant $\xi$.
\end{proof}

If we consider hyperbolic directions we know that the parabolic eigenvalues 
are separated from the hyperbolic ones and we can use the spectral 
projection associated to the group of parabolic eigenvalues to separate them 
from the hyperbolic one(s). In this case the estimate of the above theorem
is valid for the corresponding `parabolic modes' of the solution. So we
can restrict  consideration to hyperbolic eigenvalues near hyperbolic 
directions.

\subsubsection{Treatment of non-degenerate hyperbolic directions}
We consider {\em only} the for us interesting case when hyperbolic directions
 form part of a 
regular submanifold of $\S^{n-1}$ and coupling functions vanish to first order,
i.e., we assume that the corresponding coupling function 
$a_j:\S^{n-1}\supset\mathcal U\to \R$ satisfies
\begin{equation}
   \d a_j(\eta)\ne 0\quad\text{when $a_j(\eta)=0$, $\eta\in\mathcal U$.}
\end{equation}
This implies that $M_j=\{\eta\in\mathcal U\;:\;a_j(\eta)=0\}$ is regular of 
dimension $n-2$, the normal derivative $\partial_n a_j(\eta)\ne0$ never 
vanishes and $a_j(\eta)\le\epsilon$ defines a tubular neighbourhood of $M_j$ 
with a natural parameterisation. The desired dispersive estimate is related to geometric properties
of the section $\mathcal S_{(M_j)}$ of the Fresnel surface lying directly over $M_j$,
\begin{equation}
\mathcal S_{(M_j)} = \{{ \omega_j^{-1}(\eta)}\eta : \eta \in M_j \} = \mathcal S_j \cap {\mathrm{co}\,} M_j.
\end{equation}
Here $\mathrm{co}\,M_j$ denotes the cone over $M_j$.
For dimensions $n\ge 4$ we have to distinguish between different cases, depending on
whether the cross-section $\mathcal S_{(M_j)}$ of the Fresnel surface satisfies a convexity assumption or not. By the latter we mean that any intersection of $\mathcal S_j$
with a hyperplane tangent to $\mathrm{co}\, M_j$ is convex in a neighbourhood of $\mathcal S_{(M_j)}$.

If this convexity assumption is satisfied (or if $n=3$ and therefore $\dim M_j=1$), we define
the convex Sugimoto index of $\mathcal S_{(M_j)}$ as maximal order of contact of 
$\mathcal S_{(M_j)}$ with hyperplanes normal to $\mathrm{co}\, M_j$.

\begin{thm}[Hyperbolic estimate, convex case]\label{thm:4.2}
Assume that $\psi$ is supported in a sufficiently small tubular neighbourhood 
of the regular hyperbolic submanifold $M_j$ and that $\mathcal S_{(M_j)}$ satisfies the
convexity assumption. Let further $\gamma_j=\gamma(\mathcal S_{(M_j)})$ 
be defined as above.

Then the solutions to \eqref{eq:4.1} satisfy the a-priori estimate
\begin{align}
  \|\psi(\D)P_{\nu_j}(\D) V(t,\cdot)\|_q &\lesssim (1+t)^{-(\frac12+\frac{n-2}{\gamma_j})(\frac1p-\frac1q)} \| V_0\|_{p,r}
\end{align}  
for all $p\in(1,2]$, $pq=p+q$ and with Sobolev regularity $r>n(1/p-1/q)$.
\end{thm}

\begin{proof}
First, we outline the strategy of the proof. We split variables in the tubular
neighbourhood of the regular hyperbolic submanifold $M_j$, one coordinate
being the defining function $a_j(\eta)$ and the other parameterising
points on $M_j$. We have
to combine a (simple) parabolic type estimate in normal directions taking 
care of the imaginary part of the phase
with stationary phase estimates for the integration along $M_j$. The stationary phase estimate
is done first and follows the lines of \cite{Sug94}, \cite{Sug96} along with \cite[Sect. 5]{RW:11}. 

It is sufficient to show the estimate for $t\ge1$. 
We follow the treatment of Brenner \cite{Brenner75} 
and decompose the Fourier integral representing the corresponding
hyperbolic modes of the solution $V$ into dyadic pieces. For large and 
intermediate frequencies this amounts to 
estimate for all $k\in\mathbb N_0$ 
\begin{multline}
  \mathcal I_k(t) = \sup_{z\in\mathbb R^n} \bigg|
    \int_{\tilde\eta=-\epsilon}^{\tilde \eta=\epsilon} \int_{\check\eta\in M_j} 
    \int_{2^{k-1}\le |\xi|\le 2^{k+1}} 
    \mathrm e^{\i t|\xi| (z\cdot\eta + |\xi|^{-1}\nu_j(\xi))}\\ 
    \times
    p_j(\xi) \chi_k(\xi)
    |\xi|^{n-1-r}
    \d|\xi| \d \check\eta \d \tilde\eta
    \bigg|
\end{multline}
with the notation $z=x/t$, $\xi=|\xi|\eta$, $\eta\simeq(\check \eta,\tilde\eta)$
with $\check\eta\in M_j$ and $\tilde\eta = a_j(\eta)$. The amplitude $p_j(\xi)$
arises from the spectral projector $P_{\nu_j}(\D)$ and the phase
$\nu_j(\xi)$ is complex-valued with $\Im\nu_j(\xi)\sim \tilde\eta^2$ uniform in
$\xi\in\supp \chi_k$ and $k\in\mathbb N_0$. 

If $z+\nabla_\xi\nu_j(\xi)\ne0$, 
$\xi/|\xi|\in M_j$ or if $z$ is not near a direction from $M_j$,
the principle of non-stationary phase implies and gives a rapid decay. 
It suffices to restrict to $z$ corresponding to stationary points. We use the method
of stationary phase to estimate the integral over $M_j$, this can be done
uniformly over $\xi$ and $\tilde\eta$,
provided $\epsilon$ is chosen small enough and yields
an estimate of the form
\begin{equation}
 \left|  \int_{\check\eta\in M_j} \dots \d\check\eta \right| \le C t^{-\frac{n-2}{\gamma_j} }
 |\xi|^{n-1-r  -\frac{n-2}{\gamma_j} } \e^{-c\tilde\eta^2 t} 
\end{equation}
uniform in $k$ and $|\tilde\eta|\le\epsilon$. In order to obtain this estimate we apply Ruzhansky's 
multi-dimensional van der Corput lemma, \cite{Ruzh:09}, based on the uniformity of the 
Sugimoto index $\gamma(\mathcal S_j \cap \mathrm{co}\, \{ \eta : a_j(\eta) = \tilde\eta,
\; \eta\approx \check\eta\})$ for small $\tilde\eta$ and the uniform bounds on the appearing 
amplitude. Similar to \cite{RW07} the imaginary part of the phase 
can be incorporated in the estimate for the amplitude. Integration over $\tilde\eta$ yields a further decay of $t^{-1/2}$, while integrating over $\xi$ and using $|\xi|\sim 2^k$ yields
\begin{equation}
   \mathcal I_k(t) \le C t^{-\frac12-\frac{n-2}{\gamma_j}} 2^{k(n-r  -\frac{n-2}{\gamma_j})}.
\end{equation}
Hence, we need $r\ge n-\frac{n-2}{\gamma_j}$   (compared to the elasticity or wave equation with $r\ge n-\frac{n-1}{\gamma}$) to apply Brenner's argument and obtain the desired estimate for the high frequency part. The required regularity follows from using Sobolev embedding for small $t$.

The treatment of small frequencies is somewhat simpler. We do not apply
a dyadic decomposition, but still have to use a stationary phase argument
along $M_j$ combined with the behaviour of the imaginary part of the phase
away from it,
\begin{align*}
  \mathcal I(t) &= \sup_{z\in\mathbb R^n} \left|
    \int_{\tilde\eta=-\epsilon}^{\tilde \eta=\epsilon} \int_{\check\eta\in M_j} 
    \int_{ |\xi|\le 1} 
    \mathrm e^{\i t|\xi| (z\cdot\eta + |\xi|^{-1}\nu_j(\xi))}
    p_j(\xi) \chi(\xi)
    |\xi|^{n-1}
    \d|\xi| \d \check\eta \d \tilde\eta
    \right|\\
    	&\le   C t^{-\frac{n-2}{\gamma_j}} 
	   \int_{ |\xi|\le 1}  \int_{\tilde\eta=-\epsilon}^{\tilde \eta=\epsilon} 
    \mathrm e^{- c \tilde\eta^2 t |\xi|^2} |\xi| \d\tilde\eta 
    |\xi|^{n-2 - \frac{n-2}{\gamma_j}}
    \d|\xi| \le C t^{-\frac12-\frac{n-2}{\gamma_j}}.
\end{align*}
\end{proof}

Without proof we comment on the non-convex situation. If the convexity assumption is violated
we have to replace the convex Sugimoto index by a corresponding non-convex one
$\gamma_0(\mathcal S_{(M_j)})$. This
is defined as the maximum over the minimal contact orders of $\mathcal S_{(M_j)}$ with hyperplanes normal to the cone $\mathrm{co}\, M_j$, the maximum taken over all points
of $\mathcal S_{(M_j)}$. The price we have to pay for non-convexity is a loss of decay.

\begin{thm}[Hyperbolic estimate, non-convex case]
Assume that $\psi$ is supported in a sufficiently small tubular neighbourhood 
of the regular hyperbolic submanifold $M_j$ and that $\mathcal S_{(M_j)}$ does not satisfy the
convexity assumption. Let further $\widetilde{\gamma_j}=\gamma_0(\mathcal S_{(M_j)})$ 
be the non-convex Sugimoto index.

Then the solutions to \eqref{eq:4.1} satisfy the a-priori estimate
\begin{align}
  \|\psi(\D)P_{\nu_j}(\D) V(t,\cdot)\|_q &\lesssim (1+t)^{-(\frac12+\frac{1}{\widetilde{\gamma_j}})(\frac1p-\frac1q)} \| V_0\|_{p,r}
\end{align}  
for all $p\in(1,2]$, $pq=p+q$ and with Sobolev regularity $r>n(1/p-1/q)$.
\end{thm}

\subsubsection{Application to cubic and hexagonal media} 
Because of its importance later on we remark that in our applications to 
three-dimensional thermo-elasticity the manifolds $M_j$ are parts of 
circles on $\S^2$, i.e. can be seen as intersections of $\S^2$ with a 
cone. So we have to look at the corresponding sections of the 
Fresnel surface. In this case $\gamma_j$ is just the maximal order 
of tangency between the curve $\mathcal S_{(M_j)}$and its tangent lines. If the curvature of this
curve is nowhere vanishing, then $\gamma_j=2$. Furthermore,
algebraicity of $\mathcal S$ of order $6$ implies that the highest order of contact 
is $6$ and therefore $\gamma_j\in \{2,\ldots,6\}$ is the admissible range of these indices.

For cubic media there are two types of regular hyperbolic submanifolds. One is
up to symmetry given by the circle $\eta_3=0$ on $\S^2$ and the corresponding
eigenvalue is equal to $\mu$. Thus the section of the Fresnel surface
is just a circle and therefore its curvature is nowhere vanishing. Similarly,
for intersections of the Fresnel surface with the plane $\eta_2=\eta_3$
we obtain the hyperbolic eigenvalue $\varkappa = \eta_2^2(\tau-\lambda)+\eta_1^2\mu$.
It is a simple calculation\footnote{Parametrising by the angle, the hyperbolic 
eigenvalue is given by $\varkappa(\phi)=\mu+\frac{\tau-\lambda-2\mu}2\sin^2\phi$ and
it remains to check that 
$\partial_\phi^2\sqrt{\varkappa(\phi)}+\sqrt{\varkappa(\phi)}\ne0$, see \cite{Wirth:2007c} 
for such a calculation.} to show that the curvature of the corresponding 
section of the Fresnel surface is nowhere vanishing as soon as $\lambda\ne\tau$
and $\mu\ne0$. Hence, $\gamma_j=2$ in both cases.

For hexagonal media regular hyperbolic submanifolds correspond to circles on
the Fresnel surface. Again, $\gamma_j=2$.

\subsection{Cubic media in 3D} We want to discuss the derivation for estimates 
near degenerate directions by the example of cubic media in three-dimensional 
space and combine them with the general estimates from Section~\ref{sec41}.  

\subsubsection{Conic points}
The following statement resembles \cite[Thm.~1.5]{Lie04}. In \cite[Sect.~3]{Lie06} a stronger decay rate
is obtained for some conic degenerations, but they require a sufficiently bent cone and we
can not guarantee that in our case. 

\begin{thm}[Conic degeneration]\label{thm:4.3} Assume $U_1$, $U_2$ and $\theta_0$ are 
micro-locally supported in a sufficiently small conical neighbourhood of a 
conically degenerate point on $\widehat{\S^2}$. Then the corresponding 
solution to the thermo-elastic system for cubic media satisfies the a-priori estimate
\begin{equation}
   \| \sqrt{A(\D)} U(t,\cdot), U_t(t,\cdot), \theta(t,\cdot) \|_{q} \lesssim (1+t)^{-\frac12 (\frac1p-\frac1q)} 
   \| \sqrt{A(\D)} U_1, U_2, \theta_0 \|_{p,r}
\end{equation}
for $p\in(1,2]$, $pq=p+q$ and $r>3(1/p-1/q)$.
\end{thm}

\begin{proof}
The main idea is that the proof of \cite{Lie91} uses polar co-ordinates around the singularities of the Fresnel surface similar to our treatment in Section~\ref{sec3}. Stationary phase arguments are
applied in tangential direction and are uniform for small radii, while the final estimate follows after integration over the remaining variables. 

It suffices to prove the statement for $t\ge 1$, the small
time estimate is a direct consequence of Sobolev embedding theorem in combination with the obvious energy estimate.
Similar to the hyperbolic estimate discussed before, we apply a dyadic decomposition of frequency space (localised to a small conic neighbourhood of the degenerate direction). The estimate for single dyadic components follows \cite{Lie91} resp. \cite[Thm. 1.5]{Lie04}; 
the only thing we have to check is that the necessary assumptions are satisfied uniform with respect to $|\xi|$ and $k\in\mathbb N$. We consider
\begin{multline}\label{eq:conic-dyadic}
  \mathcal I_k(t) = \sup_{z\in\R^3} \bigg| \int_0^{\tilde\epsilon} \int_0^{2\pi} \int_{2^{k-1}\le|\xi|\le 2^{k+1}} 
  \e^{\i t |\xi| ( z\cdot\eta + |\xi|^{-1} \nu_j (|\xi|,\epsilon,\phi) )} \\
\times\;  p_j(|\xi|,\epsilon,\phi)
  \chi_k(|\xi|)
  |\xi|^{2-r} \d |\xi| \d\phi \epsilon \d\epsilon \bigg|,
\end{multline}
where $\eta\in\S^2$ denotes the point with polar co-ordinates $(\epsilon,\phi)$ near the conic
degenerate direction and $\xi=|\xi|\eta$. The amplitude $p_j(|\xi|,\epsilon,\phi)$ arises from the spectral projector (given in terms of the diagonaliser) constructed in the blown-up polar co-ordinates and $\chi_k(\xi)$ corresponds to the dyadic decomposition. The complex phase
$\nu_j(|\xi|,\epsilon,\phi)$ is described in Proposition~\ref{prop:3.5}. Its imaginary part is non-negative and vanishes to second order in $\epsilon=0$ as well as for three hyperbolic manifolds emanating from the conic degenerate point. Again we may treat this imaginary part as part of the
amplitude and apply stationary phase estimates for the integral with respect to $\phi$. 
As the approximation of the phase modulo $\mathcal O(\epsilon^2)$ is independent of $\phi$ and uniform
in $|\xi|$ this yields
\begin{equation}
 \left| \int_0^{2\pi} \dots \d\phi\right| \lesssim t^{-\frac12} |\xi|^{\frac32-r} \epsilon^{\frac12}
\end{equation}
uniform in $|\xi|$, $k$ and $0\le\epsilon\le\tilde\epsilon$. There is no further benefit from the imaginary part (as there can not be a lower bound with respect to $\epsilon$) and integrating with
respect to $|\xi|$ and $\epsilon$ concludes the estimate for $\mathcal I_k(t)$. 
Similarly, we estimate the small frequency part
\begin{align}
  \mathcal I(t) &= \sup_{z\in\R^3} \bigg| \int_0^{\tilde\epsilon} \int_0^{2\pi} \int_{|\xi|\le 1} 
  \e^{\i t |\xi| ( z\cdot\eta + |\xi|^{-1} \nu_j (|\xi|,\epsilon,\phi) )} 
  p_j(|\xi|,\epsilon,\phi)
  \chi(|\xi|)
  |\xi|^{2} \d |\xi| \d\phi\epsilon \d\epsilon \bigg| \notag \\&\le C t^{-\frac12},
\end{align}
such that Brenner's method again yields the desired decay estimate.
\end{proof}

\subsubsection{Uniplanar points} The treatment of uniplanar degeneracies follows
\cite{Lie04}. We have to make one further additional assumption related to
the shape of certain curves on the Fresnel surface near the degenerate point. To be precise, 
we either require that
\begin{equation}\label{eq:uni-cond}
  \Omega\cap  \mathcal S \cap \Pi \quad\text{has non-vanishing curvature}
\end{equation}
for $\Omega\subset\R^n$ an open neighbourhood of the uniplanarly degenerate point 
and for any plane $\Pi$ sufficiently close and parallel to the common tangent plane at the unode. This condition is equivalent to the technical assumption (1.12) made in \cite{Lie06}
If \eqref{eq:uni-cond} is violated, we need to consider Sugimoto indices 
$\gamma_u = \gamma(  \Omega\cap  \mathcal S \cap \Pi\subset\Pi)$, i.e., contact orders of 
these planar curves with its tangent planes combined with a uniformity assumption.
Under assumption \eqref{eq:uni-cond} the index is given by $\gamma_u=2$.

For cubic media we have to use
the statement of Proposition~\ref{prop:3:6} to determine the index $\gamma_u$. Using the notation of \eqref{eq:3:30}, it suffices to calculate the indices of the indicator curves determined 
by $\epsilon^2 \big( \mu+C\pm\sqrt{C^2\cos^2(2\phi)+D^2\sin^2(2\phi)}\big)=1$. This yields
\begin{equation}
\gamma_u\in\{2,3,4\}
\end{equation}
In the nearly isotropic case we have $\gamma_u=2$, away from it $\gamma_u=3$. Both are generic, while the borderline case with $\gamma_u=4$ is not.
The asymptotic 
construction of the eigenvalues and eigenprojections near the uniplanarly degenerate
point of Proposition~\ref{prop:3.7} yields that the assumption is satisfied uniformly
for the phase functions appearing in all dyadic components of the operator.

\begin{thm}[Uniplanar degeneration]\label{thm:4.4}
Assume $U_1$, $U_2$ and $\theta_0$ are 
micro-locally supported in a sufficiently small conical neighbourhood of a 
uniplanarly degenerate point on $\widehat{\S^2}$. Let further 
$\gamma_u$ be the index of the uniplanar point. Then the corresponding 
solution to the thermo-elastic system for cubic media satisfies the a-priori estimate
\begin{equation}
   \| \sqrt{A(\D)} U(t,\cdot), U_t(t,\cdot), \theta(t,\cdot) \|_{q} \lesssim (1+t)^{-(\frac12+\frac1{\gamma_u})(\frac1p-\frac1q)} 
   \| \sqrt{A(\D)} U_1, U_2, \theta_0 \|_{p,r}
\end{equation}
for $p\in(1,2]$, $pq=p+q$ and $r>3(1/p-1/q)$.
\end{thm}

\begin{proof}[Sketch of proof] We will sketch the major differences to the treatment of 
conic degeneracies. We will again use polar co-ordinates and estimate corresponding 
dyadic components \eqref{eq:conic-dyadic}, where now $\nu_j(|\xi|,\epsilon,\phi)$ 
is determined by Proposition~\ref{prop:3.7}. The imaginary part of $\nu_j(|\xi|,\epsilon,\phi)$ 
vanishes to third order and is of no benefit, while the real part coincides to third order with the corresponding elastic eigenvalue. This allows to use estimates from \cite{Lie04} and \cite[Sect.~4]{Lie06}, the main
difference to the previous situation is that we now use stationary phase estimates for both, 
the angular and the radial integral. The proof itself then coincides with the corresponding proof
for cubic elasticity, cf. \cite{Lie09}. 

Using a change of variables the integral is written in the new variables $\omega_j(\eta)|\xi|$ (i.e., roughly $\Re\nu_j$) and $\eta/\omega_j(\eta) \in \mathcal S_j$. In this form the phase splits and the crucial estimate is just
a Fourier transform of a density carried by the sheet of the Fresnel surface (with possible singularity in the unode). This is calculated by introducing {\em distorted} polar co-ordinates on
the surface. As level sets we use cuts of the surface by planes parallel to the common tangent plane. Then we will at first apply the method of stationary phase to the radial variable in these co-ordinates. These stationary points are non-degenerate and we use the obtained first terms in the asymptotics for a second stationary phase argument in the angular variables. The condition 
\eqref{eq:uni-cond} would imply again that stationary points are non-degenerate and we are done, while if \eqref{eq:uni-cond} is violated we use the lemma of van der Corput instead to prove the estimate.
\end{proof}

\subsubsection{Collecting the estimates} It remains to collect all the estimates into a final statement for cubic media. Parabolic directions are treated by Lemma~\ref{lem:4.1};
hyperbolic manifolds away from degenerate points are covered by Theorem~\ref{thm:4.2}. 
The remaining 24 degenerate directions fall into either of the previously discussed categories
and estimates follow from Theorem~\ref{thm:4.3} and~\ref{thm:4.4}. The resulting estimates
are collected in Table~\ref{tab1}.

\begin{table}[htbp]
\begin{center}
\begin{tabular}{|c|c|c|}
\hline&&\\ & small frequencies & large frequencies \\&&\\\hline&&\\
parabolic directions & $(1+t)^{-\frac32}$ & $\e^{-Ct}$ \\&&\\\hline&&\\
hyperbolic directions & $(1+t)^{-1}$ &$(1+t)^{-1}$\\&&\\ \hline&&\\
conic degeneracies &$(1+t)^{-\frac12}$ &$(1+t)^{-\frac12}$\\&&\\ \hline&&\\
uniplanar degeneracies & $(1+t)^{-\frac12-\frac1\gamma}$&$(1+t)^{-\frac12-\frac1\gamma}$ \\
& $\gamma\in\{2,3,4\}$ & $\gamma\in\{2,3,4\}$ \\
&&\\\hline
\end{tabular}
\end{center}
\caption{Contributions to the dispersive decay rate for cubic media.}\label{tab1}
\end{table}

\begin{cor}[Cubic decay rates]
Cubic media in three space dimensions satisfy the dispersive type estimate
\begin{equation}
   \| \sqrt{A(\D)} U(t,\cdot), U_t(t,\cdot), \theta(t,\cdot) \|_{L^q(\R^n)} \lesssim (1+t)^{-\frac12(\frac1p-\frac1q)} 
   \| \sqrt{A(\D)} U_1, U_2, \theta_0 \|_{p,r}
\end{equation}
for all data $U_1\in W^{p,r+1}(\R^3;\C^3)$, $U_2\in W^{p,r}(\R^3;\C^3)$ and 
$\theta\in W^{p,r}(\R^3)$, provided $p\in(1,2]$, $pq=p+q$ and $r>3(1/p-1/q)$. 
\end{cor}

Decay rates improve if the Fourier transform of the initial data vanishes in
the conically degenerate directions. This could be achieved by posing particular symmetry 
conditions.

\subsection{Hexagonal media}
The treatment of hexagonal media is somewhat simpler. The uniplanar degenerations trivially satisfy the assumption \eqref{eq:uni-cond} and therefore yield the decay rates specified
by the above theorem. The additionally appearing manifolds of degenerate directions 
are trivially resolved as there are smooth families of eigenprojections associated to both eigenvalues (as we stay away from the uniplanar points) and we can therefore treat the
modes separately. 

One of them is hyperbolic for all directions, we refer to it as the 
genuine hyperbolic mode. The sheet of the Fresnel surface corresponding to this
mode, i.e., to the eigenvalue $\varkappa(\eta) = \frac{\tau_1-\lambda_1}2 (\eta_1^2+\eta_2^2)+\mu \eta_3^2$ is easily seen to be strictly convex for all choices of the parameter 
and gives therefore $t^{-1}$. The proof is similar to that for the wave equation, see \cite{Brenner75}.

The parabolic modes away from the degenerate hyperbolic directions are treated as before, 
while  the remaining  degenerate 
hyperbolic manifold is treated by the estimate of Theorem~\ref{thm:4.2} with $\gamma=2$
due to rotational invariance.
The resulting estimates are collected in Table~\ref{tab2}.

\begin{table}[htpb]
\begin{center}
\begin{tabular}{|c|c|c|}
\hline&&\\ & small frequencies & large frequencies \\&&\\\hline&&\\
genuine hyperbolic mode & $(1+t)^{-1}$ &$(1+t)^{-1}$\\&&\\ \hline&&\\
parabolic modes & $(1+t)^{-3/2}$ & $\e^{-Ct}$ \\&&\\\hline&&\\
hyperbolic directions & $(1+t)^{-1}$ &$(1+t)^{-1}$\\&&\\ \hline&&\\
uniplanar degeneracies & $(1+t)^{-1}$&$(1+t)^{-1}$ \\&&\\\hline
\end{tabular}
\end{center}
\caption{Contributions to the dispersive decay rate for hexagonal media.}\label{tab2}
\end{table}

\begin{cor}[Hexagonal decay rates]
Cubic media in three space dimensions satisfy the dispersive type estimate
\begin{equation}
   \| \sqrt{A(\D)} U(t,\cdot), U_t(t,\cdot), \theta(t,\cdot) \|_{L^q(\R^n)} \lesssim (1+t)^{-(\frac1p-\frac1q)} 
   \| \sqrt{A(\D)} U_1, U_2, \theta_0 \|_{p,r}
\end{equation}
for all data $U_1\in W^{p,r+1}(\R^3;\C^3)$, $U_2\in W^{p,r}(\R^3;\C^3)$ and 
$\theta\in W^{p,r}(\R^3)$, provided $p\in(1,2]$, $pq=p+q$ and $r>3(1/p-1/q)$. 
\end{cor}

\bigskip
\noindent{\bf Acknowledgements.} The paper was inspired by many discussions with Michael Reissig and also Ya-Guang Wang, who in particular raised the interest for dispersive decay rates for thermo-elastic systems and the applied decoupling techniques to deduce them. The author is 
also grateful to Otto Liess for pointing out some of his results on decay estimates for Fourier transforms of measure carried by singular surfaces.

\end{document}